\documentclass[a4j,12pt,draft]{article}
\usepackage{amsmath}
\usepackage{amssymb}
\usepackage{amsthm}
\numberwithin{equation}{section}
\newtheorem{example}{Example}[section]

\newtheorem{theorem}[example]{Theorem}
\newtheorem{remark}[example]{Remark}
\newtheorem{lemma}[example]{Lemma}
\newtheorem{proposition}[example]{Proposition}

\DeclareMathOperator{\Ric}{Ric}
\DeclareMathOperator{\Hess}{Hess}
\DeclareMathOperator{\tr}{tr}
\DeclareMathOperator{\Rm}{Rm}
\DeclareMathOperator{\R}{R}
\begin{document}
\title{Hyperbolic thermostat and Hamilton's Harnack inequality for the Ricci flow}
\author{Tatsuhiko Kobayashi}
\date{\today}
\maketitle
\section{Introduction}
The Ricci flow was introduced by Hamilton \cite{H5} with motivation in solving the Poincar\'{e} conjecture. 
We mean by the Ricci flow a pair $(M,g)$ of a smooth manifold $M$ and an evolving Riemannian metric $g=g(x,t)$ obeying the evolution equation
\begin{equation}\label{eq:100}
\frac{\partial g(x, t)}{\partial t }=-2 \Ric _g (x, t ) \, .
\end{equation}
The Ricci flow equation is invariant under the action of the group Diff$(M)$ of diffeomorphisms of $M$, 
which means that the Ricci flow is interpreted as a gauge theory with gauge group Diff$(M)$. 
This implies, via the (contracted) second Bianchi identity, that the Ricci flow equation is not parabolic, but only weakly parabolic. 
This causes a difficulty in proving the existence of the solution to the Ricci flow for a given smooth initial metric. 
In \cite{H1}, Hamilton proved the short time existence on a closed manifold, which later was greatly simplified by DeTurck \cite{De}. \\
\\
During 1980-1990's, it was inevitable to put some assumption on the curvature of the initial or the solution metric for the study of the Ricci flow. 
Nevertheless, Hamilton \cite{H1} was able to establish a program (Hamilton program) toward proving Thurston's geometrization conjecture \cite{TH} 
by studying the time-global solution metric for the Ricci flow with arbitrary initial metric. 
Recall that Thurston's geometrization conjecture claims that any closed 3-manifold is decomposed into pieces each having one of the eight maximal model geometries. 
Here, the decomposition means first the connected sum decomposition into prime components 
and second the torus decomposition of each prime component into pieces having one of the eight maximal model geometries.\\
\\
According to Hamilton's program, the occurrence of the decomposition into the prime components is the 
effect under the formation of singularities in finite time for the Ricci flow. The difficulty 
of the proof of the ``no local collapsing property'' under the formation of singularities 
in finite time for the Ricci flow had been the major difficulty against the progress 
of Hamilton's program. In this direction, Hamilton's Harnack inequality [8], which 
compares the curvature of the Ricci flow at two points in the space-time, was the 
most prominent result in the study of the Ricci flow obtained 
in the period 1980-1990's: 
\begin{theorem}[Hamilton \cite{H1}]
Let $(M,g_{ij}(t))$ be a complete Ricci flow for $t \in (0, T] \subset \mathbb{R}^{+}$ with uniformly bounded curvature 
in the sense that there exists a constant $C$ such that $| \Rm | \leq C$. 
Suppose that $(M, g(t))$ has a weakly positive curvature operator (which means a nonnegative curvature operator). 
Let 
\begin{equation}\label{s:1}
\begin{aligned}
M_{ij}&:=\Delta R_{ij}-\frac{1}{2}\nabla _i \nabla _j \R +2R_{ikjl}R^{kl}-R_{ik}R_{j}^{k}+\frac{1}{2t}R_{ij} \, ,\\
P_{ijk}&:=\nabla _i R_{jk}-\nabla _j R_{ik} \, .
\end{aligned}
\end{equation}
Then, the Harnack expression 
\begin{equation}\label{he:1}
Z:=R_{ijkl}U^{ij}U^{kl}+2P_{ijk}U^{ij}X^k+M_{ij}X^i X^j
\end{equation}
is weakly positive for any two-forms $U^{ij}$ and any one-forms $X^k$. 
\end{theorem}
There are several known Harnack type inequalities: Li-Yau \cite{LY} for the scalar heat flow and Hamilton \cite{H2} for the mean curvature flow 
and \cite{H3} for the Ricci flow on surfaces. 
As was shown in \cite{P1}, the Harnack inequality plays an essential role in the study 
of the finite-time singularities of the Ricci flow. Since the Ricci flow equation is only 
weakly parabolic due to its diffeomorphism invariance, the concept of a self-similar 
solution, i.e., a special solution to the Ricci flow equation which evolves under a 
1-parameter family of diffeomorphisms (coupled with scalings) makes sense and 
such a solution is named a Ricci soliton. Therefore, the Ricci soliton should play 
an essential role in the proof of the Harnack inequality for the Ricci flow. 
In fact, using the equation of the expanding Ricci soliton was the basic idea of Hamilton's proof of the Harnack inequality. 
Moreover, Hamilton proved the Harnack inequality by applying the maximum principle to the Harnack 
expression which was constructed from the equation of the gradient expanding Ricci 
soliton. Hamilton's Harnack inequality is ``mysterious'' in the sense that it is proved 
under the assumption of the nonnegative curvature operator, while the expanding 
Ricci soliton generalizes the Einstein metric with negative Ricci curvature. \\
\\
It was Perelman \cite{P1} who introduced the $\mathcal{W}$-entropy (which is defined in section 5) and used its monotonicity 
under the Ricci flow to prove the no local collapsing property of the finite-time 
singularities of the Ricci flow. Moreover, Perelman was able to establish 
the propagation of the no local collapsing property to the space-time by introducing the 
reduced volume and proving their monotonicity under the Ricci flow. This way, 
Perelman was able to prove Thurston's geometrization conjecture. It was 
remarkable that Perelman combined the $\mathcal{W}$-entropy / the reduced volume and Hamilton's 
Harnack inequality in the analysis of the finite-time singularities 
(the determination of the structure of the ancient solution with no collapsing condition). \\
\\
In \cite{P1}, Perelman introduced the concept of the ``Riemannian geometric thermostat'' (which we will describe in section 3.1) and 
developed a statistical theory following the standard formalism of the statistical 
mechanics. It seems that it was the way how Perelman discovered these functionals 
having the monotonicity under the Ricci flow. The theory of the thermostat is a 
heuristic framework which produces basic quantities such as the reduced volume and 
the $\mathcal{W}$-entropy of the Ricci flow. This also gives hints (the concept of the $\mathcal{L}$-length) 
to rigorous proofs of their basic properties. By developing the $\mathcal{L}$-geometry, i.e., the 
comparison geometry based on the $\mathcal{L}$-length, Perelman was able to give rigorous 
proofs to results obtained by heuristic arguments based on the 
Riemannian geometric thermostat. \\
\\
Perelman proposed the ``correct position'' where Hamilton's 
Harnack inequality lives, i.e., the $\mathcal{L}$-geometry which emerges from the Riemannian 
geometric thermostat. Therefore, it is a conceptually interesting problem to search 
for a reason why Hamilton's Harnack inequality holds in the framework of the theory of 
the Riemmanian geometric thermostat. \\
\\
The purpose of this paper is to propose a geometric interpretation to Hamilton's Harnack inequality. 
There are several known results: Chow - Chu \cite{CC}, Chow - Knopf \cite{CK}, and Cabezas-Rivas - Topping \cite{TOP2}. 
In particular, the last paper introduced the canonical expanding soliton which unifies Hamilton's result and Brendle's result \cite{BR}. 
Moreover, Cabezas-Rivas - Topping was able to obtain some new Harnack inequalities for the Ricci flow 
from the view point of the canonical expanding soliton. \\
\\
In this paper, we introduce a variant of Riemannian geometric thermostat namely the hyperbolic thermostat, which we now describe. 
\begin{theorem}\label{th:1}
Let $(M,g_{ij}(t))$ be a complete Ricci flow for $t \in (0, T] \subset \mathbb{R}^{+}$ with uniformly bounded curvature, 
and $(\mathbb{H}^N,g_{\alpha \beta})$ be an $N$-dimensional hyperbolic space with constant sectional curvature $-\frac{1}{2N}$ for $N \in \mathbb{N}$. 
We define $\tilde{M}=M \times \mathbb{H}^N \times (0,T ]$, and a metric $\tilde{g}$ on $\tilde{M}$ as follows:
\begin{align*}
\tilde{g}_{ij}=g_{ij} \, , \quad
\tilde{g}_{\alpha \beta}= t g_{\alpha \beta } \, ,\quad
\tilde{g}_{00}=\R -\frac{N}{2t } \, , \quad \tilde{g}_{i \alpha} =\tilde{g}_{0 i} = \tilde{g}_{0 \alpha}=0 \, ,
\end{align*}
where $i, j$ are coordinate indices on the $M$ factor, $\alpha , \beta$ are coordinate indices on the $\mathbb{H}^N$ factor, 
0 represent the index of the time coordinate $t$, and $\R$ is the scalar curvature with respect to the metric $g_{ij}$. 
Then, the Ricci tensor $\widetilde{\Ric}$ is equal to zero up to errors of order $\frac{1}{N}$, i.e., $(\tilde{M}, \tilde{g})$ is Ricci flat up to errors of order $\frac{1}{N}$. 
Moreover, Hamilton's Harnack expression appears as the full curvature operator with respect to the metric $\tilde{g}$ up to errors of order $\frac{1}{N}$. 
\end{theorem}
We will compute the curvature of the hyperbolic thermostat in section 3.2. 
Since the full curvature tensor of the hyperbolic thermostat gives rise to the exact Hamilton's Harnack expression, 
we expect to recover Hamilton's Harnack inequality by applying the preservation principle under the Ricci flow. 
It is well known that the positivity of some curvature is preserved under the Ricci flow. 
In \cite{TOP2}, Cabezas-Rivas and Topping recovered Hamilton's Harnack inequality by proving an appropriate preservation principle. 
In the present paper, we will recover Hamilton's Harnack inequality 
by proving the positivity of the curvature operator of the restricted hyperbolic thermostat: 
\begin{theorem}[Main Theorem]
Let $(M,g_{ij}(t))$ be a complete Ricci flow for $t \in (0, T] \subset \mathbb{R}^{+}$ with uniformly bounded curvature, and 
assume that the manifold $(M, g_{ij}(t))$ has a weakly positive curvature operator. 
Then, the manifold $(\bar{M}, \bar{g}_{ab}(t))$, which is obtained by restricting the hyperbolic thermostat to a submanifold, has weakly positive curvature operator. 
\end{theorem}
Details for the manifold $(\bar{M}, \bar{g}_{ab}(t))$ can be found in section 3. 
In the course of proving Theorem 1.3, we will establish new geometric interpretations to several crucial quantities in Hamilton's original proof of the Harnack inequality. 
We prove Theorem 1.3 by adapting Hamilton's original argument to our setting of the hyperbolic thermostat. 
In particular, we imitate Hamilton's theorem [11, Theorem 4.1] to deduce the inequality in Lemma $\ref{l:3}$ (see Remark $\ref{la:1}$). \\ 
\\
Hyperbolic thermostat and the canonical expanding soliton induced Harnack inequality. 
On the other hand, the canonical shrinking soliton introduced by \cite{TOP3} and \cite{TOP4} recover some results on the Ricci flow 
which were discovered by Perelman \cite{P1} (e.g. the monotonicity of $\mathcal{W}$-entropy). 
In fact, the monotonicity of $\mathcal{W}$-entropy under the Ricci flow is induced from the view point of Riemannian geometric thermostat 
by formally applying a comparison theorem for the total scalar curvature (See section 5). 
Hence, we hope that all of results from the canonical solitons are interpreted as the results from thermostats. \\
\\
This paper is organized as follows. 
In section 2 we recall the basic set up for the Riemannian geometry and basic properties of the Ricci flow. 
In section 3 we see that the Harnack expression for the Ricci flow $g(t)$ appears as the full curvature tensor of the hyperbolic thermostat metric $\tilde{g}(t)$ and 
derive the differential equations of the components of curvature tensor. 
In section 4 we state and prove the Main Theorem which is equivalent to Hamilton's Harnack inequality, 
by applying the maximum principle along a submanifold $\bar{M}$ in $\tilde{M}$ equipped with a degenerate metric $\bar{g}$ as a section of $Sym^2 (T^{\ast} \bar{M})$. 
Here, $\tilde{M}$ is a potentially infinite dimensional manifold and $\bar{M}$ is a ($\dim M +1$)-dimensional manifold. 
In section 5 we prove the comparison theorem for the total scalar curvature on geodesic sphere which holds if the manifold is complete and Ricci flat. 
The monotonicity of $\mathcal{W}$-entropy is recovered from a view point of Riemannian's geometric thermostat by formally applying this comparison theorem. 

\section{Preparation}
In this section, we define fundamental quantities for a Riemannian manifold. 
Moreover, we deduce some basic properties of the Ricci flow. 
In this paper, we adopt the convention of curvatures in \cite{TOP} as stated below.

\subsection{Riemannian geometry}
Let $(M, g)$ be a Riemannian manifold, and $X, Y, Z, W$ vector fields independent of time $t$ on $M$. 
We define the full curvature tensor $\Rm$, Ricci curvature $\Ric$, and scalar curvature $\R$: 
\begin{align*}
&\Rm (X,Y)Z:=\nabla _Y \nabla _X Z - \nabla _X \nabla _Y Z+\nabla _{[X,Y]} Z \, ,\\
&\Rm (X,Y,Z,W):=\langle{\Rm (X,Y)Z},{W}\rangle \, , \\
&\Ric (X,Y):=\tr \Rm (X, \cdot , Y, \cdot )\, , \\
&\R:=\tr \Ric \, ,
\end{align*}
where $\langle{\cdot}, {\cdot}\rangle$ is the inner product with respect to $g$, 
$\nabla$ is the Levi-Civita connection with respect to $g$, and $[X, Y] :=\nabla _X Y-\nabla _Y X$ is the Lie bracket. 
These components are expressed as
\begin{align*}
&R_{ijk}^{\, \, \, \, \, \, \, \, l} \frac{\partial}{\partial x^l}=\Rm \Bigl(\frac{\partial}{\partial x^i},\frac{\partial}{\partial x^j}\Bigl)\frac{\partial}{\partial x^k} \, ,\\
&R_{ijkl}=g_{pl}R_{ijk}^{\, \, \, \, \, \, \, \, p}\, , \\
&R_{ij}=g^{kl}R_{ikjl}\ ,\\
&\R =g^{ij}R_{ij}\, ,
\end{align*}
where $(x^1, \cdots , x^n)$ are local coordinates on $M$ and $\{ \frac{\partial}{\partial x^i}\} ^{n}_{i=1}$ is a local frame of $TM$ consisting of coordinate vector fields. 
Note that we adopt the Einstein summation convention.\\
\\
The full curvature tensor has some symmetry properties:
\begin{equation*}
\begin{aligned}
&R_{ijkl}=-R_{jikl}=-R_{ijlk}=R_{klij}\, , \\
&R_{ijkl}+R_{jkil}+R_{kijl}=0\, ,\\
&\nabla _i R_{jklm} +\nabla _j R_{kilm}+\nabla _k R_{ijlm}=0 \, .
\end{aligned}
\end{equation*}
The second and third equations are well-known as the first and the second Bianchi identities, respectively. 
By taking the trace of the second Bianchi identity, we have 
\begin{equation}\label{eq:22}
\nabla ^l R_{lijk}=\nabla _j R_{ki}- \nabla _k R_{ji}
\end{equation}
where $\nabla ^{i}:= g^{ij} \nabla _{j}$ . We take the trace again to get the twice contracted second Bianchi identity:
\begin{equation}\label{eq:1}
\nabla ^{i}R_{ij}=\frac{1}{2}\nabla _{j} \R \, .
\end{equation}
We note that the first and second Bianchi identities are equivalent to the diffeomorphism invariance of the curvature tensor \cite{Kaz}. 
For arbitrary tensor $A$, we have the equation which deduce the commutation equation for Levi-Civita connection:
\begin{equation}\label{re:1}
\nabla ^2 _{X,Y} A-\nabla ^2 _{Y,X} A =-\Rm (X, Y)A \, ,
\end{equation}
where $\nabla ^2 _{X,Y}:=\nabla _X \nabla _Y -\nabla _{\nabla _X Y}$. 
This equation is called the Ricci identity. 
The following two examples for the tensors $R_{ij}$ and $\nabla _p R_{ij}$ are important to know how geometric quantities evolve under the Ricci flow. 
\begin{equation}\label{eq:51}
\begin{aligned}
\nabla _m \nabla _n R_{ij}- \nabla _n \nabla _m R_{ij}=&R_{mni}^{\,\,\,\,\,\,\,\,\,\,\,p}R_{pj}+R_{mnj}^{\,\,\,\,\,\,\,\,\,\,\,p}R_{ip} \, ,\\
\nabla _m \nabla _n \nabla _p R_{ij}-\nabla _n \nabla _m \nabla _p R_{ij}=&R_{mnp}^{\,\,\,\,\,\,\,\,\,\,\,\, q}\nabla _q R_{ij}+R_{mni}^{\,\,\,\,\,\,\,\,\,\,\,\,\, q}\nabla _p R_{qj}+R_{mnj}^{\,\,\,\,\,\,\,\,\,\,\,\, q}\nabla _p R_{iq} \, .
\end{aligned}
\end{equation}
We use these identities in the proof of the equations $(\ref{eq:21})$. 
By the symmetric properties of the full curvature tensor, we naturally define the curvature operator $\mathcal{R} \colon \wedge ^2 T^{\ast} M \longrightarrow \wedge ^2 T^{\ast} M$ as follows:
\begin{align*}
\Rm (X,Y,Z,W)=\langle{\mathcal{R}(X\wedge Y)}, {Z\wedge W}\rangle \, .
\end{align*}
We call the curvature operator is positive (resp. weakly positive) when $\Rm ({U}, {U}) >0$ (resp. $\geq 0$) for any two-forms ${U}$. 
Our sign convention is that $R_{ijij}>0$ on the round sphere. 

\subsection{The evolution equations for the curvature under the Ricci flow}
In this section, we will see how geometric quantities evolve
when the metric evolves under the Ricci flow. 
Details can be found, for instance, in [19, Chapter 2]. \\
\\
Let $(M,g(t))$ be a Ricci flow, i.e.
\begin{equation*}
\frac{\partial}{\partial t } g_{ij}(t)=-2 R_{ij} (t ) 
\end{equation*}
as described in $(\ref{eq:100})$. 
The curvature tensor evolves as follows: 
\begin{align}\label{eq:112}
\frac{\partial}{\partial t}\Rm (X,Y,W,Z)=(\Delta \Rm)(X,Y,W,Z)+Q(X,Y,W,Z)+F(X,Y,W,Z)
\end{align}
where
\begin{equation*}
\begin{aligned}
Q(X,Y,W,Z):=&2(B(X,Y,W,Z)-B(X,Y,Z,W) \\
&+B(X,W,Y,Z)-B(X,Z,Y,W)) \, , \\
F(X,Y,W,Z):=&-\Ric (\Rm (X,Y)W,Z)+\Ric (\Rm (X,Y)Z,W)\\
&-\Ric (\Rm (W,Z)X,Y)+\Ric (\Rm (W,Z)Y,X)\, , \\
B(X,Y,W,Z):=&\langle{\Rm (X,\cdot , Y , \cdot )},{\Rm (W,\cdot , Z, \cdot )}\rangle \, ,
\end{aligned}
\end{equation*}
and $\Delta$ is Laplace-Beltarami operator $\Delta := g^{ij} \nabla _{i} \nabla _{j}$\, . 
This equation is also expressed as 
\begin{equation}\label{e:7}
\begin{aligned}
\frac{\partial}{\partial t}R_{ijkl}=&\Delta R_{ijkl}+2R_{imjn}R^{\,\,\, m\,\, n}_{k \,\,\,\, l}-2R_{imjn}R^{\,\,\, m\,\, n}_{l \,\,\,\, k}\\
&+2R_{imkn}R^{\,\,\, m\,\, n}_{j \,\,\,\, l}-2R_{imln}R^{\,\,\, m\,\, n}_{j \,\,\,\, k}\\
&-R^{m}_{l} R_{ijkm}+R^{m}_{k} R_{ijlm}-R^{m} _{j}R_{klim}+R^m _i R_{kljm} \, ,
\end{aligned}
\end{equation}
where $R_i{}^m{}_j{}^n:=g^{jm}g^{ln}R_{ijkl}$ and $R^{m}_{i}:=g^{jm}R_{ij}$. 
By taking the trace, we have the evolution equation for the Ricci tensor: 
\begin{equation}\label{e:5}
\frac{\partial}{\partial t}R_{ij}=\Delta R_{ij} -2 R^{k}_{i}R_{kj}+2R_{ikjl}R^{kl} \, .
\end{equation}
Moreover, by taking the trace again, we have the evolution equation for the scalar curvature:
\begin{equation*}
\frac{\partial R}{\partial t}=\Delta \R +2R^{ij} R_{ij} \, .
\end{equation*}
We set 
\begin{equation}\label{e:6}
\begin{aligned}
P_{ijk}&:=\nabla _{i}R_{jk}-\nabla _{j}R_{ik}\, , \\
M_{ij}&:=\Delta R_{ij} +2R_{ikjl}R^{kl}-R^{k}_{i}R_{kj}-\frac{1}{2}\nabla _{i}\nabla _{j}\R+\frac{1}{2t}R_{ij}\, ,
\end{aligned}
\end{equation}
as described in $(\ref{s:1})$
The tensors $P_{ijk}$ , $M_{ij}$ , and $R_{ijkl}$ are essentially parts of the curvature tensor $\tilde{R}_{abcd}$ on the hyperbolic thermostat $(\tilde{M}, \tilde{g}_{ab} )$ 
which we will define in section 3. These tensors are first introduced by Hamilton as Hamilton's Harnack expression. 
We now compute the deformation of these tensors when the metric $g(t)$ evolves under the Ricci flow. 
The following lemma will be used to compute the evolution equation for $\tilde{R}_{abcd}$ . 
\begin{lemma}[Hamilton \cite{H1}]\label{lem:1}
Let $(M, g(t))$ be a Ricci flow, and $P_{ijk}$ , $M_{ij}$ tensors defined as above. 
Then, we have 
\begin{equation}\label{eq:21}
\begin{aligned}
\Bigl( \frac{\partial}{\partial t} - \Delta \Bigl) P_{ijk}=&2R_{imjn}P^{mn} _{\,\,\,\,\,\,\,\,\,k}+2R_{imkn}P^{m \,\, n}_{\,\,\,\,\,\, j} +2R_{jmkn}P_{i} ^{\,\, mn}-2R^{m}_n \nabla _{m}R_{ijk}^{\,\,\,\,\,\,\,\,n}\\
&+ R^m_i P_{mjk} +R^m_j P_{imk} +R^m_k P_{ijm} \, , \\
\Bigl( \frac{\partial}{\partial t}-\Delta \Bigl) M_{ij}=&2R_{imjn}M^{mn}+2R^{mn}[\nabla _m P_{nij}+\nabla _{m}P_{nji}]\\
&+2P_{imn}P_{j}^{\,\,mn}-4P_{imn}P_{j}^{\,\,nm}+2R^{mn}R_{m}^{l}R_{injl}-\frac{1}{2t^2}R_{ij}\\
&+R^m_i M_{mj} +R^m_j M_{im} 
\end{aligned}
\end{equation}
\end{lemma}
The first equation is shown by using the formula $(\ref{eq:51})$ and 
$$P_{ijk}+P_{jki}+P_{kij}=0 \, .$$ 
To get the second equation, we use the formula 
$$M_{ij}=\nabla ^p P_{pij}+R_{ikjl}R^{kl}+\frac{1}{2t}R_{ij} \, .$$ 
\begin{remark}\label{rem:1}
{\rm In \cite{H1}, Hamilton shows this lemma by using the vector field $D_t$ on the orthonormal frame bundle of the Ricci flow where $D_t$ is defined by $D_t:=\frac{\partial }{\partial t}+R_{b} ^{a}\nabla ^b _a$. 
Here, $\nabla ^b _a$ is obtained by modifying $\frac{\partial}{\partial t}$ by vertical vector field so that $D_t$ is tangent to the orthonormal frame bundle of the Ricci flow. 
The advantage of using $D_t$ instead of $\frac{\partial}{\partial t}$ is that the frame $\{e_a\}$ which is a local orthonormal frame at $t=0$ behaves like a local orthonormal frame for all time $t$ under the Ricci flow. 
However, we don't use $D_t$ in this paper because we will deal with the parameter $t$ not only as a time but also as a part of local coordinate system 
$t:= x^0$ on $(\tilde{M}, \tilde{g}_{ab})$ .}
\end{remark}

\section{Riemannian geometric thermostat}

In this section, we canonically construct hyperbolic thermostat from arbitrary Ricci flow following Perelman's construction of Riemannian geometric thermostat. 
Moreover, we derive the evolution equation of the curvature tensor. 

\subsection{Spherical thermostat}
Here, we describe Riemannian geometric thermostat which was introduced by Perelman in \cite{P1}. 
We only recall the full curvature tensor of the thermostat in this section. 
We will see the other properties in section 5. 

\begin{theorem}[Perelman \cite{P1}]
Let $(M,g_{ij}(\tau ))$ be a complete backward Ricci flow (i.e.$\frac{\partial g_{ij}}{\partial \tau }=2 R_{ij} (\tau )$ ) for $\tau \in (0, T] \subset \mathbb{R}^{+}$ with uniformly bounded curvature, 
and $(\mathbb{S}^N,g_{\alpha \beta})$ a sphere with constant sectional curvature $\frac{1}{2N}$ for $N \in \mathbb{N}$. 
We define  $\hat{M}=M \times \mathbb{S}^N \times (0,T ]$, and a metric $\hat{g}$ on $\hat{M}$ as follows:
\begin{align*}
\hat{g}_{ij}=g_{ij}\, , \quad
\hat{g}_{\alpha \beta}= \tau g_{\alpha \beta } \, , \quad
\hat{g}_{00}=\R+\frac{N}{2\tau } \, , \,  \hat{g}_{i \alpha} =\hat{g}_{0 i} = \hat{g}_{0 \alpha}=0 \, ,
\end{align*}
where $i, j$ are coordinate indices on the $M$ factor, $\alpha , \beta$ are coordinate indices on the $\mathbb{S}^{N}$ factor, 
0 represent the index of the scale coordinate $\tau$, $\R$ is the scalar curvature with respect to the metric $g_{ij}$.
Then, the Ricci tensor $\widehat{\Ric}$ is equal to zero up to errors of order $\frac{1}{N}$, i.e., $(\hat{M}, \hat{g})$ is Ricci flat up to errors of order $\frac{1}{N}$.
\end{theorem}
Ricci flatness can be shown by the fundamental computation in the same way as the proof of Theorem 3.2. 
We see that the Hamilton's Harnack expression appears as the components of full curvature tensor $\hat{R}_{ijkl}$ , $\hat{R}_{ij0k}$ , and $\hat{R}_{0i0j}$ by setting $\tau = -t$ up to errors of order $\frac{1}{N}$. 
Hence, one may expect to prove Hamilton's Harnack inequality from the view point of the thermostat by using the basic property of the Ricci flow that the weak positivity of curvature operator is preserved 
(maximum principle). 
However, since $\tau >0$ and Hamilton's Harnack inequality holds with $t>0$, the Harnack expression which appears in the thermostat has opposite sign at terms 
where $\tau$ appears. 
This motivates the hyperbolic thermostat we will introduce in the next section. 
\\
\subsection{Hyperbolic thermostat}
We now describe the construction of the hyperbolic thermostat, which we already described in Theorem 1.2:
\begin{theorem}
Let $(M,g_{ij}(t))$ be a complete Ricci flow for $t \in (0, T] \subset \mathbb{R}^{+}$ with uniformly bounded curvature, 
and $(\mathbb{H}^N,g_{\alpha \beta})$ be a N-dimensional hyperbolic space with constant sectional curvature $-\frac{1}{2N}$ for $N \in \mathbb{N}$. 
We define  $\tilde{M}=M \times \mathbb{H}^N \times (0,T ]$, and a metric $\tilde{g}$ on $\tilde{M}$ as follows:
\begin{align*}
\tilde{g}_{ij}=g_{ij} \, , \quad
\tilde{g}_{\alpha \beta}= t g_{\alpha \beta } \, , \quad
\tilde{g}_{00}=\R-\frac{N}{2t } \, , \quad  \tilde{g}_{i \alpha} =\tilde{g}_{0 i} = \tilde{g}_{0 \alpha}=0 \, ,
\end{align*}
where $i, j$ are coordinate indices on the $M$ factor, $\alpha , \beta$ are coordinate indices on the $\mathbb{H}^N$ factor, 
0 represent the the index of the time coordinate $t$, $\R$ is the scalar curvature with respect to the metric $g_{ij}$.
Then, $\widetilde{\Ric}$ is equal to zero up to errors of order $\frac{1}{N}$, i.e., $(\tilde{M}, \tilde{g})$ is Ricci flat up to errors of order $\frac{1}{N}$. 
Moreover, Hamilton's Harnack expression appears as the full curvature operator with respect to the metric $\tilde{g}_{ab}$ up to errors of order $\frac{1}{N}$. 
\end{theorem}

Note that $\tilde{M}$ has potentially infinite dimension 
because later we take the limit $N \rightarrow \infty$, 
and $\tilde{g}_{0 0}$ is negative for sufficiently large $N>0$. 
Hence, $\tilde{g}$ is a Lorentzian metric on $\tilde{M}$ in this situation. 
When we see the tensor $\tilde{g}$ as a metric on $T^{\ast}\tilde{M}$ rather than $T \tilde{M}$, it degenerates in a limit $N \rightarrow \infty$. 
Then, this metric converges to a weakly positive definite tensor in this limit. \\
\\
We now prove Theorem 3.2 (which is equivalent to Theorem 1.2) by using basic computations in Riemannian geometry. 
Note that $\tilde{M}$ has potentially infinite dimension 
because later we take the limit $N \rightarrow \infty$, 
and $\tilde{g}_{0 0}$ is negative for sufficiently large $N>0$. 
Hence, $\tilde{g}$ is a Lorentzian metric on $\tilde{M}$ in this situation. 
When we see the tensor $\tilde{g}$ as a metric on $T^{\ast}\tilde{M}$ rather than $T \tilde{M}$, it degenerates in a limit $N \rightarrow \infty$. 
Then, this metric converges to a weakly positive definite tensor in this limit. 
\begin{proof}[proof of Theorem 3.2]
We use the fundamental formula for the Christoffel symbols, 
\begin{equation*}
\tilde{\Gamma}^a_{bc}=\frac{1}{2}\tilde{g}^{ad}\Bigl( \frac{\partial \tilde{g}_{cd}}{\partial x^b} +\frac{\partial \tilde{g}_{bd}}{\partial x^c} -\frac{\partial \tilde{g}_{bc}}{\partial x^d} \Bigl)  
\end{equation*}
to compute all kinds of $\tilde{\Gamma}^{a}_{bc}$ , where the indices $a, b, c$ represent either $i$ (the indices for $M$), $0$ (the index for $\mathbb{R}^+$) or $\alpha$ (the indices for $\mathbb{H}^{N}$): 
\begin{align*}
\tilde{\Gamma}^{0}_{0 0}&=\frac{1}{2}\tilde{g}^{00}\Bigl( \frac{\partial \tilde{g}_{00}}{\partial x^0}\Bigl) =\frac{1}{2}\Bigl( \R-\frac{N}{2t}\Bigl )^{-1}\Bigl( \frac{\partial \R}{\partial t} + \frac{N}{2t^2}\Bigl) \, , \\
\tilde{\Gamma}^{0}_{i 0}&=\frac{1}{2}\tilde{g}^{00}\Bigl( \frac{\partial \tilde{g}_{00}}{\partial x^i}+\frac{\partial \tilde{g}_{i0}}{\partial x^0}-\frac{\partial \tilde{g}_{i0}}{\partial x^0}\Bigl)
=\frac{1}{2}\Bigl( \R-\frac{N}{2t} \Bigl) ^{-1}\frac{\partial \R}{\partial x^{i}} \, , \\
\tilde{\Gamma}^{i}_{0 0}&=\frac{1}{2}\tilde{g}^{ij}\Bigl(\frac{\partial \tilde{g}_{0j}}{\partial x^0}+\frac{\partial \tilde{g}_{0j}}{\partial x^0}-\frac{\partial \tilde{g}_{00}}{\partial x^j}\Bigl)=-\frac{1}{2}g^{ij} \frac{\partial \R}{\partial x^{j}}\, , \\
\tilde{\Gamma}^{i}_{j 0}&=\frac{1}{2}\tilde{g}^{ik}\Bigl(\frac{\partial \tilde{g}_{0k}}{\partial x^j}+\frac{\partial \tilde{g}_{ik}}{\partial x^0}-\frac{\partial \tilde{g}_{j0}}{\partial x^k}\Bigl)=-R^{i} _{j} \, , \\
\tilde{\Gamma}^{0}_{i j}&=\frac{1}{2}\tilde{g}^{00}\Bigl(\frac{\partial \tilde{g}_{j0}}{\partial x^i}+\frac{\partial \tilde{g}_{i0}}{\partial x^j}-\frac{\partial \tilde{g}_{ij}}{\partial x^0}\Bigl)=\Bigl( \R-\frac{N}{2t}\Bigl) ^{-1}R_{ij} \, , \\
\tilde{\Gamma}^{i}_{j k}&=\frac{1}{2}\tilde{g}^{il}\Bigl(\frac{\partial \tilde{g}_{kl}}{\partial x^j}+\frac{\partial \tilde{g}_{il}}{\partial x^k}-\frac{\partial \tilde{g}_{jk}}{\partial x^l}\Bigl)=\Gamma ^{i} _{j k} \, , \\ 
\tilde{\Gamma}^{\alpha}_{\beta 0}&=\frac{1}{2}\tilde{g}^{\alpha \gamma}\Bigl(\frac{\partial \tilde{g}_{0 \gamma}}{\partial x^{\beta}}+\frac{\partial \tilde{g}_{\beta \gamma}}{\partial x^0}
-\frac{\partial \tilde{g}_{\beta 0}}{\partial x^\gamma}\Bigl)
=\frac{1}{2t} \delta ^{\alpha} _{\beta} \, , \\
\tilde{\Gamma}^{0}_{\alpha \beta}&=\frac{1}{2}\tilde{g}^{00}\Bigl(\frac{\partial \tilde{g}_{\beta 0}}{\partial x^{\alpha}}+\frac{\partial \tilde{g}_{\alpha 0}}{\partial x^{\beta}}-\frac{\partial \tilde{g}_{\alpha \beta}}{\partial x^0}\Bigl)
=-\frac{1}{2}\Bigl( \R -\frac{N}{2t} \Bigl) ^{-1} g_{\alpha \beta} \, , \\
\tilde{\Gamma}^{\alpha}_{\beta \gamma}&=\frac{1}{2}\tilde{g}^{\alpha \delta}\Bigl( \frac{\partial \tilde{g}_{\gamma \delta}}{\partial x^{\beta}}+\frac{\partial \tilde{g}_{\beta \delta}}{\partial x^{\gamma}}
-\frac{\partial \tilde{g}_{\beta \gamma}}{\partial x^{\delta}}\Bigl)=\Gamma ^{\alpha}_{\beta \gamma} \, ,\\
\end{align*}
at a point $(\tilde{x}^{a} ) := (x^{i}, x^{\alpha}, x^{0})\in \tilde{M}$, where $R^{i}_{j}:=g^{ik} R_{jk}$\, , $x^0:= t$\, , $\Gamma ^{i} _{j k} $ is Christoffel symbols of $g_{ij}$ at the point $x^i \in M$, and 
$\Gamma ^{\alpha} _{\beta \gamma} $ is Christoffel symbols of $g_{\alpha \beta}$ at the point $x^{\alpha} \in \mathbb{H}^{N}$. 
Since 
\begin{align*}
\tilde{\nabla}_{\frac{\partial}{\partial x^0}} \frac{\partial}{\partial x^j}=\tilde{\Gamma}^a_{0j}\frac{\partial}{\partial x^a}
=-R^i_j \frac{\partial}{\partial x^i}+\frac{1}{2}\Bigl (\R -\frac{N}{2t} \Bigl) ^{-1}\frac{\partial \R}{\partial x^{i}}\frac{\partial}{\partial x^0} \, , 
\end{align*}
the following proposition holds:
\begin{proposition}
The orthonormal frame $\{ \frac{\partial}{\partial x^j} \} _{j=1}^n \subset \{ \frac{\partial}{\partial x^a}\} _{a=1} ^{n+N+1}$ at $t=0$ remains orthonormal after the time passes up to errors of order $\frac{1}{N}$
\end{proposition}
From Proposition 3.3, we see that the covariant derivative $\tilde{\nabla}_{\frac{\partial}{\partial x^0}}$ plays a role like the vector field $D_t$ as mentioned in Remark 2.2. \\
\\
By the definition of $\tilde{g}_{ab}$, the other components of Christoffel symbol are clearly vanished:
\begin{align*}
\tilde{\Gamma}^{\alpha }_{0 0}&=\frac{1}{2}\tilde{g}^{\alpha \beta} \Bigl( \frac{\partial \tilde{g}_{0\beta}}{\partial x^0}+\frac{\partial \tilde{g}_{0\beta}}{\partial x^0}-\frac{\partial \tilde{g}_{00}}{\partial x^{\beta}}\Bigl) =0 \, , \\
\tilde{\Gamma}^{0}_{\alpha 0}&=\frac{1}{2}\tilde{g}^{00}\Bigl( \frac{\partial \tilde{g}_{00}}{\partial x^{\alpha}}+\frac{\partial \tilde{g}_{\alpha 0}}{\partial x^0}-\frac{\partial \tilde{g}_{\alpha 0}}{\partial x^0}\Bigl) =0\, , \\
\tilde{\Gamma}^{i}_{\alpha 0}&=\frac{1}{2}\tilde{g}^{ij} \Bigl( \frac{\partial \tilde{g}_{j0}}{\partial x^{\alpha}}+\frac{\partial \tilde{g}_{\alpha i}}{\partial x^0}-\frac{\partial \tilde{g}_{\alpha 0}}{\partial x^j}\Bigl) =0\, , \\
\tilde{\Gamma}^{\alpha}_{i 0}&=\frac{1}{2}\tilde{g}^{\alpha \beta}\Bigl( \frac{\partial \tilde{g}_{\beta 0}}{\partial x^i}+\frac{\partial \tilde{g}_{i\beta}}{\partial x^0}-\frac{\partial \tilde{g}_{i0}}{\partial x^{\beta}}\Bigl) =0\, , \\
\tilde{\Gamma}^{0}_{i \alpha}&=\frac{1}{2}\tilde{g}^{00}\Bigl( \frac{\partial \tilde{g}_{\alpha 0}}{\partial x^i}+\frac{\partial \tilde{g}_{i0}}{\partial x^{\alpha}}-\frac{\partial \tilde{g}_{i\alpha}}{\partial x^0}\Bigl) =0\, , \\
\tilde{\Gamma}^{i}_{\alpha \beta}&=\frac{1}{2}\tilde{g}^{ij}\Bigl( \frac{\partial \tilde{g}_{\beta j}}{\partial x^{\alpha}}+\frac{\partial \tilde{g}_{\alpha j}}{\partial x^{\beta}}-\frac{\partial \tilde{g}_{\alpha \beta}}{\partial x^j}\Bigl) =0\, , \\
\tilde{\Gamma}^{\alpha}_{i j}&=\frac{1}{2}\tilde{g}^{\alpha \beta}\Bigl( \frac{\partial \tilde{g}_{j \beta}}{\partial x^i}+\frac{\partial \tilde{g}_{i\beta}}{\partial x^j}-\frac{\partial \tilde{g}_{ij}}{\partial x^{\beta}}\Bigl) =0\, , \\
\tilde{\Gamma}^{i}_{j \alpha}&=\frac{1}{2}\tilde{g}^{ik}\Bigl( \frac{\partial \tilde{g}_{\alpha k}}{\partial x^j}+\frac{\partial \tilde{g}_{jk}}{\partial x^{\alpha}}-\frac{\partial \tilde{g}_{j\alpha}}{\partial x^k}\Bigl) =0\, , \\
\tilde{\Gamma}^{\alpha}_{\beta i}&=\frac{1}{2}\tilde{g}^{\alpha \gamma}
\Bigl( \frac{\partial \tilde{g}_{i\gamma}}{\partial x^{\beta}}+\frac{\partial \tilde{g}_{\beta \gamma}}{\partial x^i}-\frac{\partial \tilde{g}_{\beta i}}{\partial x^{\gamma}}\Bigl) =0 \, .
\end{align*}
We compute the Ricci tensor with respect to $\tilde{g}$ by taking the trace of the curvature tensor. 
First, by using the standard formula for the curvature tensor:
\begin{equation*}
\tilde{R}_{abcd}=\tilde{g}_{df}\Bigl( \frac{\partial \tilde{\Gamma}^f_{ac}}{\partial x^b}-\frac{\partial \tilde{\Gamma}^f_{bc}}{\partial x^a}+\tilde{\Gamma}^e_{ac}\tilde{\Gamma}^f_{be}-\tilde{\Gamma}^e_{bc}\tilde{\Gamma}^f_{ae}\Bigl) \, , 
\end{equation*}
we have 
\begin{equation}\label{eq:13}
\begin{aligned}
\tilde{R}_{ijkl}=&\tilde{g}_{lm}\Bigl( \frac{\partial \tilde{\Gamma} ^m _{ik}}{\partial x^j}-\frac{\partial \tilde{\Gamma} ^m _{jk}}{\partial x^i}+\tilde{\Gamma} ^e _{ik}\tilde{\Gamma} ^m _{je}-\tilde{\Gamma} ^e _{jk}\tilde{\Gamma} ^m _{ie}\Bigl) \\
=&R_{ijkl}+\tilde{g}_{lm}(\tilde{\Gamma} ^0 _{ik}\tilde{\Gamma} ^m _{j0}-\tilde{\Gamma} ^0 _{jk}\tilde{\Gamma} ^m _{i0})\\
=&R_{ijkl}+g_{lm}\Bigl\{ \Bigl( \R -\frac{N}{2t}\Bigl) ^{-1}R_{ik}(-R^m_{j})-\Bigl( \R -\frac{N}{2t}\Bigl) ^{-1}R_{jk}(-R^m_{i})\Bigl\} \\
=&R_{ijkl}-\Bigl( \R -\frac{N}{2t}\Bigl) ^{-1}(R_{ik}R_{jl}+R_{jk}R_{il}) \, ,
\end{aligned}
\end{equation}
\begin{equation}\label{eq:113}
\begin{aligned}
\tilde{R}_{ij0k}=&\tilde{g}_{km}
\Bigl( \frac{\partial \tilde{\Gamma} ^m _{i0}}{\partial x^j}-\frac{\partial \tilde{\Gamma} ^m _{j0}}{\partial x^i}+\tilde{\Gamma} ^e _{i0}\tilde{\Gamma} ^m _{je}-\tilde{\Gamma} ^e _{j0}\tilde{\Gamma} ^m _{ie}\Bigl) \\
=&g_{km}\Bigl\{ \frac{\partial}{\partial x^j}(-R_i ^m)-\frac{\partial}{\partial x^i}(-R_j ^m)+\Gamma ^m _{jn}(-R_i ^n)-\Gamma ^m _{in} (- R_j ^n)\Bigl\} \\
&+g_{km}\Bigl\{ \frac{1}{2}\Bigl( \R -\frac{N}{2t}\Bigl) ^{-1}\frac{\partial \R}{\partial x^i} (-R_j ^m)-\frac{1}{2}\Bigl( \R -\frac{N}{2t}\Bigl) ^{-1}\frac{\partial \R}{\partial x^j} (-R_i ^m)\Bigl\} \\
=&\nabla _i R_{jk}-\nabla _j R_{ik}+\frac{1}{2}\Bigl( \R -\frac{N}{2t}\Bigl) ^{-1} \Bigl( \frac{\partial \R}{\partial x^j}R_{ik}-\frac{\partial \R}{\partial x^i}R_{jk}\Bigl) \, ,
\end{aligned}
\end{equation}
\begin{equation}\label{eq:1113}
\begin{aligned}
\tilde{R}_{i0j0}=&\tilde{g}_{00}\Bigl( \frac{\partial \tilde{\Gamma} ^0 _{ij}}{\partial x^0}-\frac{\partial \tilde{\Gamma} ^0 _{0j}}{\partial x^i}+\tilde{\Gamma} ^e _{ij}\tilde{\Gamma} ^0 _{0e}-\tilde{\Gamma} ^e _{0j}\tilde{\Gamma} ^0 _{ie}\Bigl) \\
=&\tilde{g}_{00}\Bigl[ \frac{\partial}{\partial t}\Bigl\{ \Bigl( \R -\frac{N}{2t}\Bigl) ^{-1} R_{ij} \Bigr\} -\frac{\partial}{\partial x^i} \Bigl( \frac{1}{2}\Bigl( \R -\frac{N}{2t}\Bigl) ^{-1}\frac{\partial \R}{\partial x^j} \Bigl)\\
&+ \tilde{\Gamma} ^m _{ij} \frac{1}{2} \Bigl( \R -\frac{N}{2t}\Bigl) ^{-1} \frac{\partial \R}{\partial x^m} +\Bigl( \R -\frac{N}{2t}\Bigl) ^{-1}\frac{1}{2}\Bigl( \R -\frac{N}{2t}\Bigl) ^{-1}(\frac{\partial \R}{\partial t}+\frac{N}{2t^2}\Bigl) \\
&+R^m _j\Bigl( \R -\frac{N}{2t}^{-1}\Bigl) R_{im}-\frac{1}{2}\Bigl( \R -\frac{N}{2t}\Bigl) ^{-1} \frac{\partial \R}{\partial x^j}\frac{1}{2}\Bigl( \R -\frac{N}{2t}\Bigl) ^{-1}\frac{\partial \R}{\partial x^i } \Bigr]\\
=&-\Bigl( \R -\frac{N}{2t}\Bigl) ^{-1}\Bigl( \frac{\partial \R}{\partial t}+\frac{N}{2t^2}\Bigl) R_{ij}+\frac{\partial R_{ij}}{\partial t}+\frac{1}{2}\Bigl( \R -\frac{N}{2t}\Bigl) ^{-1}\frac{\partial \R}{\partial x^i}\frac{\partial \R}{\partial x^j }\\
&-\frac{1}{2}\frac{\partial ^2 \R}{\partial x^i \partial x^j}+\frac{1}{2} \Gamma ^m _{ij} \frac{\partial \R}{\partial x^m}+\frac{1}{2}\Bigl( \R -\frac{N}{2t}\Bigl) ^{-1}\Bigl( \frac{\partial \R}{\partial t}+\frac{N}{2t^2}\Bigl) R_{ij}\\
&+R^m _j R_{im}-\frac{1}{4}\Bigl( \R -\frac{N}{2t}\Bigl) ^{-1}\frac{\partial \R}{\partial x^i}\frac{\partial \R}{\partial x^j}\\
=&\frac{\partial R_{ij}}{\partial t}-\frac{1}{2}\Bigl( \frac{\partial ^2 \R}{\partial x^i \partial x^j}-\Gamma ^m _{ij} \frac{\partial \R}{\partial x^m}\Bigl) +R^m_j R_{im}\\
&-\frac{1}{2}\Bigl( \R -\frac{N}{2t}\Bigl) ^{-1}\Bigl( \frac{\partial \R}{\partial t}+\frac{N}{2t^2}\Bigl) R_{ij} +\frac{1}{4}\Bigl( \R -\frac{N}{2t}\Bigl) \frac{\partial \R}{\partial x^i}\frac{\partial \R}{\partial x^j}\\
=&\Delta R_{ij} +2R_{ikjl}R^{kl}-\frac{1}{2}\nabla _i \nabla _j \R -R^m_j R_{im}\\
&-\frac{1}{2}\Bigl( \R -\frac{N}{2t}\Bigl) ^{-1}\Bigl( \frac{\partial \R}{\partial t}+\frac{N}{2t^2}\Bigl) R_{ij}+\frac{1}{4}(\R -\frac{N}{2t}\Bigl) ^{-1}\frac{\partial \R}{\partial x^i}\frac{\partial \R}{\partial x^j}\, , 
\end{aligned}
\end{equation}
where $\nabla$ is the Levi-Civita connection associated to the metric $g_{ij}$ . Here, we have used the evolution equation for the Ricci tensor ($\ref{e:5}$).   \\
\\
We see that the other components vanish up to errors of order $\frac{1}{N}$:
\begin{equation*}
\begin{aligned}
\tilde{R}_{\alpha \beta \gamma \delta}=&\tilde{g}_{\delta \varepsilon}\Bigl( \frac{\partial \tilde{\Gamma} ^{\varepsilon} _{\alpha \gamma}}{\partial x^{\beta}}-\frac{\partial \tilde{\Gamma} ^{\varepsilon} _{\beta \gamma}}{\partial x^{\alpha}}
+\tilde{\Gamma} ^e _{\alpha \gamma}\tilde{\Gamma} ^{\varepsilon} _{\beta e}-\tilde{\Gamma} ^e _{\beta \gamma}\tilde{\Gamma} ^{\varepsilon} _{\alpha e}\Bigl) \\
=&tR_{\alpha \beta \gamma \delta}+\frac{1}{4}\Bigl( \R -\frac{N}{2t}\Bigl) ^{-1}(g_{\beta \gamma}g_{\alpha \delta}-g_{\alpha \gamma}g_{\beta \delta})\\
=&-\frac{1}{4}\Bigl\{ \frac{2t}{N}+\Bigl( \R -\frac{N}{2t}\Bigl) ^{-1} \Bigl\} ( g_{\alpha \gamma}g_{\beta \delta}-g_{\beta \gamma}g_{\alpha \delta})\, , \\
\tilde{R}_{\alpha \beta \gamma i}=&\tilde{g}_{ij}\Bigl( \frac{\partial \tilde{\Gamma} ^{j} _{\alpha \gamma}}{\partial x^{\beta}}-\frac{\partial \tilde{\Gamma} ^{j} _{\beta \gamma}}{\partial x^{\alpha}}
+\tilde{\Gamma} ^e _{\alpha \gamma}\tilde{\Gamma} ^{j} _{\beta e}-\tilde{\Gamma} ^e _{\beta \gamma}\tilde{\Gamma} ^{j} _{\alpha e}\Bigl) =0\, , \\
\tilde{R}_{\alpha \beta \gamma 0}=&\tilde{g}_{00}\Bigl( \frac{\partial \tilde{\Gamma} ^{0} _{\alpha \gamma}}{\partial x^{\beta}}-\frac{\partial \tilde{\Gamma} ^{0} _{\beta \gamma}}{\partial x^{\alpha}}
+\tilde{\Gamma} ^e _{\alpha \gamma}\tilde{\Gamma} ^{0} _{\beta e}-\tilde{\Gamma} ^e _{\beta \gamma}\tilde{\Gamma} ^{0} _{\alpha e}\Bigl) \\
=&\tilde{g}_{00}\Bigl\{ -\frac{1}{2} \Bigl( \R -\frac{N}{2t}\Bigl) ^{-1}\frac{\partial g_{\alpha \gamma}}{\partial x^{\beta}}+\frac{1}{2}\Bigl( \R -\frac{N}{2t}\Bigl) ^{-1}\frac{\partial g_{\beta \gamma}}{\partial x^{\alpha}}\\
&-\frac{1}{2}\Gamma ^{\delta} _{\alpha \gamma}\Bigl( \R -\frac{N}{2t}\Bigl) ^{-1}g_{\beta \delta}+\frac{1}{2}\Gamma ^{\delta} _{\beta \gamma}\Bigl( \R -\frac{N}{2t}\Bigl) ^{-1}g_{\alpha \delta} \Bigl\} \\
=&-\frac{1}{2}\Bigl( \frac{\partial g_{\alpha \gamma}}{\partial x^{\beta}}-\frac{\partial g_{\beta \gamma}}{\partial x^{\alpha}}+\Gamma ^{\delta} _{\alpha \gamma}g_{\beta \delta}-\Gamma ^{\delta} _{\beta \gamma}g_{\alpha \delta}\Bigl) \\
=&-\frac{1}{2}\Bigl\{ \frac{\partial g_{\alpha \gamma}}{\partial x^{\beta}}-\frac{\partial g_{\beta \gamma}}{\partial x^{\alpha}}
+\frac{1}{2}g^{\delta \varepsilon}\Bigl( \frac{\partial g_{\gamma \varepsilon}}{\partial x^{\alpha}} +\frac{\partial g_{\alpha \varepsilon}}{\partial x^{\gamma}} -\frac{\partial g_{\alpha \gamma}}{\partial x^{\varepsilon}}\Bigl) g_{\beta \delta}\\
&-\frac{1}{2}g^{\delta \varepsilon}\Bigl( \frac{\partial g_{\gamma \varepsilon}}{\partial x^{\beta}} +\frac{\partial g_{\beta \varepsilon}}{\partial x^{\gamma}} -\frac{\partial g_{\beta \gamma}}{\partial x^{\varepsilon}}\Bigl) \Bigl\}\\
=&0\, , \\
\tilde{R}_{\alpha \beta ij}=&\tilde{g}_{jk}\Bigl( \frac{\partial \tilde{\Gamma} ^{k} _{\alpha i}}{\partial x^{\beta}}-\frac{\partial \tilde{\Gamma} ^{k} _{\beta i}}{\partial x^{\alpha}}
+\tilde{\Gamma} ^e _{\alpha i}\tilde{\Gamma} ^{k} _{\beta e}-\tilde{\Gamma} ^e _{\beta i}\tilde{\Gamma} ^{k} _{\alpha e}\Bigl) =0\, , \\
\tilde{R}_{\alpha \beta 0 i}=&\tilde{g}_{ij}\Bigl( \frac{\partial \tilde{\Gamma} ^{j} _{\alpha 0}}{\partial x^{\beta}}-\frac{\partial \tilde{\Gamma} ^{j} _{\beta 0}}{\partial x^{\alpha}}
+\tilde{\Gamma} ^e _{\alpha 0}\tilde{\Gamma} ^{j} _{\beta e}-\tilde{\Gamma} ^e _{\beta 0}\tilde{\Gamma} ^{j} _{\alpha e}\Bigl) =0\, , \\
\tilde{R}_{\alpha i \beta j}=&\tilde{g}_{jk}\Bigl( \frac{\partial \tilde{\Gamma} ^{k} _{\alpha \beta}}{\partial x^i}-\frac{\partial \tilde{\Gamma} ^{k} _{i \beta}}{\partial x^{\alpha}}
+\tilde{\Gamma} ^e _{\alpha \beta}\tilde{\Gamma} ^{k} _{i e}-\tilde{\Gamma} ^e _{i \beta}\tilde{\Gamma} ^{k} _{\alpha e}\Bigl) \\
=&\frac{1}{2}\Bigl( \R -\frac{N}{2t}\Bigl) ^{-1}g_{\alpha \beta}R_{ij}\, , \\
\tilde{R}_{\alpha 0 \beta i}=&\tilde{g}_{ij}\Bigl( \frac{\partial \tilde{\Gamma} ^{j} _{\alpha \beta}}{\partial x^{0}}-\frac{\partial \tilde{\Gamma} ^{j} _{0 \beta}}{\partial x^{\alpha}}
+\tilde{\Gamma} ^e _{\alpha \beta}\tilde{\Gamma} ^{j} _{0 e}-\tilde{\Gamma} ^e _{0 \beta}\tilde{\Gamma} ^{j} _{\alpha e}\Bigl) \\
=&\frac{1}{4}\Bigl( \R -\frac{N}{2t}\Bigl) ^{-1}g_{\alpha \beta}\frac{\partial \R}{\partial x^{i}}\, , 
\end{aligned}
\end{equation*}
\begin{equation*}
\begin{aligned}
\tilde{R}_{\alpha 0 \beta 0}=&\tilde{g}_{00}\Bigl( \frac{\partial \tilde{\Gamma} ^{0} _{\alpha \beta}}{\partial x^{0}}-\frac{\partial \tilde{\Gamma} ^{0} _{0 \beta}}{\partial x^{\alpha}}
+\tilde{\Gamma} ^e _{\alpha \beta}\tilde{\Gamma} ^{0} _{0 e}-\tilde{\Gamma} ^e _{0 \beta}\tilde{\Gamma} ^{0} _{\alpha e}\Bigl) \\
=&\tilde{g}_{00}\Bigl\{ \frac{1}{2} \Bigl( \R -\frac{N}{2t}\Bigl) ^{-2}\Bigl( \frac{\partial \R}{\partial t}+\frac{N}{2t^2}\Bigl) g_{\alpha \beta} \\
&-\frac{1}{2}\Bigl( \R -\frac{N}{2t}\Bigl) ^{-1}g_{\alpha \beta}+\frac{1}{2}\Bigl( \R -\frac{N}{2t}\Bigl) ^{-1}\Bigl( \frac{\partial \R}{\partial t}+\frac{N}{2t^2}\Bigl) \\
&-\frac{1}{2t}\delta^{\gamma} _{\beta}\Bigl( -\frac{1}{2}\Bigl( \R -\frac{N}{2t}\Bigl) ^{-1}g_{\alpha \gamma}\Bigl) \Bigl\} \\
=&\frac{1}{4}\Bigl( \R -\frac{N}{2t}\Bigl) ^{-1}\Bigl( \frac{\partial \R}{\partial t}+\frac{\R}{t}\Bigl) g_{\alpha \beta}\, , \\
\tilde{R}_{\alpha i j k }=&\tilde{g}_{kl}\Bigl( \frac{\partial \tilde{\Gamma} ^{l} _{\alpha j}}{\partial x^{i}}-\frac{\partial \tilde{\Gamma} ^{l} _{ij}}{\partial x^{\alpha}}
+\tilde{\Gamma} ^e _{\alpha j}\tilde{\Gamma} ^{l} _{i e}-\tilde{\Gamma} ^e _{ij}\tilde{\Gamma} ^{l} _{\alpha e}\Bigl) =0\, , \\
\tilde{R}_{\alpha i j 0}=&\tilde{g}_{00}\Bigl( \frac{\partial \tilde{\Gamma} ^{0} _{\alpha j}}{\partial x^{i}}-\frac{\partial \tilde{\Gamma} ^{0} _{ij}}{\partial x^{\alpha}}
+\tilde{\Gamma} ^e _{\alpha j}\tilde{\Gamma} ^{0} _{i e}-\tilde{\Gamma} ^e _{ij}\tilde{\Gamma} ^{0} _{\alpha e}\Bigl) =0\, , \\
\tilde{R}_{\alpha 0 jk}=&\tilde{g}_{jk}\Bigl( \frac{\partial \tilde{\Gamma} ^{k} _{\alpha i}}{\partial x^{0}}-\frac{\partial \tilde{\Gamma} ^{k} _{0i}}{\partial x^{\alpha}}
+\tilde{\Gamma} ^e _{\alpha i}\tilde{\Gamma} ^{k} _{0 e}-\tilde{\Gamma} ^e _{0i}\tilde{\Gamma} ^{k} _{\alpha e}\Bigl) =0\, , \\
\tilde{R}_{\alpha 0 j 0}=&\tilde{g}_{00}\Bigl( \frac{\partial \tilde{\Gamma} ^{0} _{\alpha i}}{\partial x^{0}}-\frac{\partial \tilde{\Gamma} ^{0} _{0i}}{\partial x^{\alpha}}
+\tilde{\Gamma} ^e _{\alpha i}\tilde{\Gamma} ^{0} _{0i}-\tilde{\Gamma} ^e _{0i}\tilde{\Gamma} ^{0} _{\alpha e}\Bigl) =0 \, .
\end{aligned}
\end{equation*}
Since the Ricci tensors are defined by $\tilde{R}_{ab}:=\tilde{g}^{cd}\tilde{R}_{acbd}$\, , we have 
\begin{align*}
\tilde{R}_{00}=&\tilde{g}^{ij}\tilde{R}_{0i0j}+\tilde{g}^{\alpha \beta}\tilde{R}_{0\alpha 0 \beta}\\
=&g^{ij}\Bigl\{ \Delta R_{ij} +2R_{ikjl}R^{kl}-\frac{1}{2}\nabla _i \nabla _j \R -R^m_j R_{im}\\
&-\frac{1}{2}\Bigl( \R -\frac{N}{2t}\Bigl) ^{-1}\Bigl( \frac{\partial \R}{\partial t}+\frac{N}{2t^2}\Bigl) R_{ij}+\frac{1}{4}\Bigl( \R -\frac{N}{2t}\Bigl) ^{-1}\frac{\partial \R}{\partial x^i}\frac{\partial \R}{\partial x^j}\Bigl\} \\
&+\frac{1}{t}g^{\alpha \beta}\frac{1}{4t}g_{\alpha \beta}+\frac{1}{4}\Bigl( \R -\frac{N}{2t}\Bigl) ^{-1}\Bigl( \frac{\partial \R}{\partial t}+\frac{N}{2t^2}\Bigl) g_{\alpha \beta}\\
=&\frac{1}{2}\Delta \R +|\Ric |^2 _{g} +\frac{1}{4}\Bigl( \R -\frac{N}{2t}\Bigl) ^{-1}|\nabla \R |_{g} +\frac{N}{4t^2}-\frac{1}{2}\frac{\partial \R}{\partial t}-\frac{N}{4t^2}\\
=&\frac{1}{4}\Bigl( \R -\frac{N}{2t}\Bigl) ^{-1} |\nabla \R |_g\, , 
\end{align*}
\begin{align*}
\tilde{R}_{0i}=&\tilde{g}^{jk}\tilde{R}_{0jik}+\tilde{g}^{\alpha \beta}\tilde{R}_{0\alpha i \beta}\\
=&\tilde{g}^{jk}\tilde{R}_{ik0j}+\tilde{g}^{\alpha \beta}\tilde{R}_{\alpha 0 \beta i}\\
=&g^{jk}\Bigl\{ \nabla _i R_{kj}-\nabla _k R_{ij}+\frac{1}{2}\Bigl( \R -\frac{N}{2t}\Bigl) ^{-1}\Bigl( \frac{\partial \R}{\partial x^k}R_{ij}-\frac{\partial \R}{\partial x^i}R_{kj}\Bigl) \Bigl\} \\
&+\frac{1}{t}g^{\alpha \beta}\frac{1}{4}\Bigl( \R -\frac{N}{2t}\Bigl) ^{-1}\frac{\partial \R}{\partial x^i}g_{\alpha \beta}\\
=&\nabla _i \R -g^{jk}\nabla _k R_{ji}+\frac{1}{2}\Bigl( \R -\frac{N}{2t}\Bigl) ^{-1}\Bigl( \frac{\partial \R}{\partial x^j}R^j _{i}-\frac{\partial \R}{\partial x^i}\R \Bigl)\\
&+\frac{N}{4t}\Bigl( \R -\frac{N}{2t}\Bigl) ^{-1}\frac{\partial \R}{\partial x^i}\,  \quad [\rm{from \, the \, contracted \, Bianchi \, identity \, (\ref{eq:1})}] \\
=&\nabla _i \R -\frac{1}{2}\nabla _i \R+\frac{1}{2}\Bigl( \R-\frac{N}{2t}\Bigl) ^{-1}\frac{\partial \R}{\partial x^j}R^j _{i}-\frac{1}{2}\Bigl( \R-\frac{N}{2t}\Bigl) ^{-1}\frac{\partial \R}{\partial x^i}\Bigl( \R-\frac{N}{2t}\Bigl) \\
=&\frac{1}{2}\Bigl( \R-\frac{N}{2t}\Bigl) ^{-1}\frac{\partial \R}{\partial x^{j}}R^{j}_{i}\, , \\
\tilde{R}_{0\alpha}=&\tilde{g}^{00}\tilde{R}_{0 i \alpha j}+\tilde{g}^{\beta \gamma}\tilde{R}_{0 \beta \alpha \gamma}=0\, , \\
\tilde{R}_{i \alpha}=&\tilde{g}^{00}\tilde{R}_{i0\alpha 0}+\tilde{g}^{jk}\tilde{R}_{ij\alpha k}+\tilde{g}^{\beta \gamma}\tilde{R}_{i\beta \alpha \gamma}=0\, , \\
\tilde{R}_{\alpha \beta}=&\tilde{g}^{00}\tilde{R}_{\alpha 0 \beta 0}+\tilde{g}^{ij}\tilde{R}_{\alpha i \beta j}+\tilde{g}^{\gamma \delta}\tilde{R}_{\alpha \gamma \beta \delta}\\
=&\frac{1}{4t}\Bigl( \R -\frac{N}{2t}\Bigl) ^{-1}g_{\alpha \beta}+\frac{1}{4}\Bigl( \frac{\partial \R}{\partial t}+\frac{N}{2t^2}\Bigl) \Bigl( \R -\frac{N}{2t}\Bigl) ^{-2}g_{\alpha \beta}\\
&+\frac{1}{2}\Bigl( \R -\frac{N}{2t}\Bigl) ^{-1}\R g_{\alpha \beta}
-\frac{1}{2 N}\Bigl( N g_{\alpha \beta}-g_{\alpha \beta}\Bigl) \\
&-\frac{1}{4t}\Bigl( \R -\frac{N}{2t}\Bigl) ^{-1}\Bigl( N g_{\alpha \beta}-g_{\alpha \beta}\Bigl) \\
=&\Bigl( \R -\frac{N}{2t}\Bigl) ^{-2}g_{\alpha \beta}\Bigl\{ \frac{\R}{4t}-\frac{N}{8t^2}+\frac{1}{4}\Bigl( \frac{\partial \R}{\partial t}+\frac{N}{2t^2}\Bigl) +\frac{\R ^2}{2}-\frac{\R}{4t}-\frac{\R ^2}{2}+\frac{\R N}{2t}\\
&-\frac{N^2}{8t^2}+\frac{\R ^2}{2N}-\frac{\R N}{2Nt}+\frac{N}{8t^2}-\frac{\R N}{4t}+\frac{N^2}{8t^2}+\frac{\R}{4t}-\frac{N}{8t^2}\Bigl\} \\
&=\frac{1}{4}\Bigl( \R -\frac{N}{2t}\Bigl) ^{-2}\Bigl( \frac{\partial \R}{\partial t}+\frac{2\R ^2}{N}\Bigl) g_{\alpha \beta}\, , \\
\end{align*}

\begin{align*}
\tilde{R}_{ij}=&\tilde{g}^{00}\tilde{R}_{i0j0}+\tilde{g}^{kl}\tilde{R}_{ikjl}+\tilde{g}^{\alpha \beta}\tilde{R}_{i\alpha j \beta}\\
=&\Bigl( \R -\frac{N}{2t}\Bigl) ^{-1}(\Delta R_{ij} +2R_{ikjl}R^{kl}-\frac{1}{2}\nabla _i \nabla _j \R -R^m_j R_{im})\\
&-\frac{1}{2}\Bigl( \R -\frac{N}{2t}\Bigl) ^{-2}(\frac{\partial \R}{\partial t}R_{ij}-\frac{1}{2}\frac{\partial \R}{\partial x^i}\frac{\partial \R}{\partial x^j}+\frac{N}{2t^2}R_{ij})\\
&-\Bigl( \R -\frac{N}{2t}\Bigl) ^{-1}(R_{ij}\R -R^m _i R_{mj})+R_{ij}+\frac{N}{2t}(\R -\frac{N}{2t})^{-1}R_{ij}\\
=&\Bigl( \R -\frac{N}{2t}\Bigl) ^{-1}(\Delta R_{ij}+2R_{ikjl}R^{kl}-\frac{1}{2}\nabla _{i}\nabla _{j}\R )\\
&+\frac{1}{2}\Bigl( \R -\frac{N}{2t}\Bigl) ^{-2}\Bigl( \frac{1}{2}\frac{\partial \R }{\partial x^{i}}\frac{\partial \R}{\partial x^{j}}-\Bigl( \frac{\partial \R}{\partial t}+\frac{N}{2t^2}\Bigl) R_{ij}\Bigl) \, .
\end{align*}
If we take a limit as $N\rightarrow \infty$, then we see all of the components of Ricci tensor converge to zero. 
Moreover, one can see that the norm (could be negative since the metric $\tilde{g}$ is Lorentzian metric) of Ricci tensor is also zero up to errors of order $\frac{1}{N}$. \\
\\
From the equations (\ref{eq:13}), (\ref{eq:113}), (\ref{eq:1113}), when we take the limit $N\rightarrow \infty$, we have

\begin{equation*}
\begin{aligned}
\tilde{R}_{ijkl}\longrightarrow &R_{ijkl} \, ,\\
\tilde{R}_{ij0k}\longrightarrow &\nabla _i R_{jk}-\nabla _j R_{ik} \, ,\\
\tilde{R}_{i0j0}\longrightarrow &\Delta R_{ij} +2R_{ikjl}R^{kl}-R^m_j R_{im}-\frac{1}{2}\nabla _i \nabla _j \R +\frac{1}{2t} R_{ij} \, .
\end{aligned}
\end{equation*}
Hence, Hamilton's Harnack expression (defined in Theorem 1.1) appears as the components of the full curvature tensor.  
\end{proof}
\begin{remark}
{\rm Note that our calculation results in this section hold when we consider the components of the tensors as a function on $\tilde{M}$. 
We cannot regard the right hand side of these equations as the tensor of $\tilde{M}$. 
For example, we have
\begin{equation*}
\tilde{R}_{i\alpha j \beta}=\tilde{R}_{\alpha i j \beta}
\end{equation*}
by the symmetric property of the curvature tensor of course. 
However, in the right hand side of the following equation
\begin{equation*}
\tilde{R}_{\alpha i \beta j}=\frac{1}{2}\Bigl( \R -\frac{N}{2t}\Bigl) ^{-1}g_{\alpha \beta}R_{ij}\, , 
\end{equation*}
the indices $i$ and $\alpha$ is not interchanged, because this operation is not well-defined. Hence, we should be careful when we compute tensors on $\tilde{M}$}. 
\end{remark}
Now we consider the differential equations for the coefficients of the curvature $\tilde{R}_{ijkl}$\, , $\tilde{R}_{ij0k}$\, , $\tilde{R}_{0i0j}$\, . 
The result of the following computations are not necessary to prove the Main Theorem, 
but it is useful to understand the structure of the hyperbolic thermostat. \\ 
\\
At first, we compute the Laplacian of $\tilde{R}_{abcd}$\, . We define the Laplacian with respect to $\tilde{g}$\, , 
\begin{equation*}
\tilde{\Delta}:=\tilde{g}^{ab}\tilde{\nabla}_a \tilde{\nabla}_b \, . 
\end{equation*}
In fact, the term including $\tilde{g}^{00}$ goes to zero up to errors of order $\frac{1}{N}$ since $g^{00}$ is of magnitude $O(\frac{1}{N})$ . \\
\\
We often write 
$$A \underset{N}{=} B $$
that means ``A is equal to B up to errors of order $\frac{1}{N}$'' .\\
\\
Using the formula for the tensor $\tilde{A}\in \otimes ^4 T^{\ast}\tilde{M}$, 
\begin{equation}\label{eq:121}
\begin{aligned}
\tilde{\nabla}_k \tilde{\nabla}_l \tilde{A}_{abcd}=&\frac{\partial ^2 \tilde{A}_{abcd}}{\partial x^k \partial x^l}-\frac{\partial \tilde{\Gamma ^m _{la}}}{\partial x^k}\tilde{A}_{mbcd}
-\frac{\partial \tilde{\Gamma ^m _{lb}}}{\partial x^k}\tilde{A}_{amcd}-\frac{\partial \tilde{\Gamma ^m _{lc}}}{\partial x^k}\tilde{A}_{abmd}-\frac{\partial \tilde{\Gamma ^m _{ld}}}{\partial x^k}\tilde{A}_{abcm}\\
&-\tilde{\Gamma} ^m _{la}\frac{\partial }{\partial x^k}\tilde{A}_{mbcd}-\tilde{\Gamma} ^m _{lb}\frac{\partial }{\partial x^k}\tilde{A}_{amcd}
-\tilde{\Gamma} ^m _{lc}\frac{\partial }{\partial x^k}\tilde{A}_{abmd}-\tilde{\Gamma} ^m _{ld}\frac{\partial }{\partial x^k}\tilde{A}_{abcm}\\
&-\tilde{\Gamma} ^m _{kl}\frac{\partial }{\partial x^m}\tilde{A}_{abcd}-\tilde{\Gamma} ^m _{ka}\frac{\partial }{\partial x^l}\tilde{A}_{mbcd}-\tilde{\Gamma} ^m _{kb}\frac{\partial }{\partial x^l}\tilde{A}_{amcd}\\
&-\tilde{\Gamma} ^m _{kc}\frac{\partial }{\partial x^l}\tilde{A}_{abmd}-\tilde{\Gamma} ^m _{kd}\frac{\partial }{\partial x^l}\tilde{A}_{abcm}\\
&+\tilde{\Gamma} ^m _{kl} \tilde{\Gamma} ^n _{ma}\tilde{A}_{nbcd}+\tilde{\Gamma} ^m _{kl} \tilde{\Gamma} ^n _{mb}\tilde{A}_{ancd}
+\tilde{\Gamma} ^m _{kl} \tilde{\Gamma} ^n _{mc}\tilde{A}_{abnd}+\tilde{\Gamma} ^m _{kl} \tilde{\Gamma} ^n _{md}\tilde{A}_{abcn}\\
&+\tilde{\Gamma} ^m _{ka} \tilde{\Gamma} ^n _{lm}\tilde{A}_{nbcd}+\tilde{\Gamma} ^m _{ka} \tilde{\Gamma} ^n _{lb}\tilde{A}_{mncd}
+\tilde{\Gamma} ^m _{ka} \tilde{\Gamma} ^n _{lc}\tilde{A}_{mbnd}+\tilde{\Gamma} ^m _{ka} \tilde{\Gamma} ^n _{ld}\tilde{A}_{mbcn}\\
&+\tilde{\Gamma} ^m _{kb} \tilde{\Gamma} ^n _{la}\tilde{A}_{nmcd}+\tilde{\Gamma} ^m _{kb} \tilde{\Gamma} ^n _{lm}\tilde{A}_{ancd}
+\tilde{\Gamma} ^m _{kb} \tilde{\Gamma} ^n _{lc}\tilde{A}_{amnd}+\tilde{\Gamma} ^m _{kb} \tilde{\Gamma} ^n _{ld}\tilde{A}_{amcn}\\
&+\tilde{\Gamma} ^m _{kc} \tilde{\Gamma} ^n _{la}\tilde{A}_{nbmd}+\tilde{\Gamma} ^m _{kc} \tilde{\Gamma} ^n _{lb}\tilde{A}_{anmd}
+\tilde{\Gamma} ^m _{kc} \tilde{\Gamma} ^n _{lm}\tilde{A}_{abnd}+\tilde{\Gamma} ^m _{kc} \tilde{\Gamma} ^n _{ld}\tilde{A}_{abmn}\\
&+\tilde{\Gamma} ^m _{kd} \tilde{\Gamma} ^n _{la}\tilde{A}_{nbcm}+\tilde{\Gamma} ^m _{kd} \tilde{\Gamma} ^n _{lb}\tilde{A}_{ancm}
+\tilde{\Gamma} ^m _{kd} \tilde{\Gamma} ^n _{lc}\tilde{A}_{abnm}+\tilde{\Gamma} ^m _{kd} \tilde{\Gamma} ^n _{lm}\tilde{A}_{abcn} \, ,
\end{aligned}
\end{equation}
and taking the trace with respect to the metric $\tilde{g}$, we have 
\begin{equation}\label{eq:105}
\begin{aligned}
\tilde{\Delta}\tilde{R}_{ijkl}=&\tilde{g}^{ab}[\tilde{\nabla}_a \tilde{\nabla}_b \tilde{R}_{ijkl}]\\
\underset{N}{=}&g^{mn}[\tilde{\nabla}_m \tilde{\nabla}_n \tilde{R}_{ijkl}]+\tilde{g}^{\alpha \beta} [\tilde{\nabla}_{\alpha} \tilde{\nabla}_{\beta} \tilde{R}_{ijkl}]\\
\underset{N}{=}&g^{mn}[ \nabla _m \nabla _n \tilde{R}_{ijkl}]-\tilde{g}^{\alpha \beta}\Bigl[\tilde{\Gamma} ^0 _{\alpha \beta} \tilde{\nabla}_{0} \tilde{R}_{ijkl}\Bigl] \\
\underset{N}{=}&\Delta R_{ijkl}-\tilde{\nabla}_{0} \tilde{R}_{ijkl} \, ,
\end{aligned}
\end{equation}
\begin{align*}
\tilde{\Delta}\tilde{R}_{ij0k}=&g^{mn}[\tilde{\nabla}_m \tilde{\nabla}_n \tilde{R}_{ij0k}]+\tilde{g}^{\alpha \beta}[\tilde{\nabla}_{\alpha} \tilde{\nabla}_{\beta} \tilde{R}_{ij0k}]\\
=&g^{mn}\Bigl[ \nabla _m \nabla _n \tilde{R}_{ij0k}-\frac{\partial \tilde{\Gamma ^p _{n0}}}{\partial x^m}\tilde{R}_{ijpk}
-\tilde{\Gamma} ^p _{n0}\frac{\partial }{\partial x^m}\tilde{R}_{ij0k}
-\tilde{\Gamma} ^p _{m0}\frac{\partial }{\partial x^n}\tilde{R}_{ijpk}\\
&+\tilde{\Gamma} ^p _{mn} \tilde{\Gamma} ^q _{p0}\tilde{R}_{ijqk}
+\tilde{\Gamma} ^p _{mi} \tilde{\Gamma} ^q _{n0}\tilde{R}_{pjqk}
+\tilde{\Gamma} ^p _{mj} \tilde{\Gamma} ^q _{n0}\tilde{R}_{ipqk}
+\tilde{\Gamma} ^p _{m0} \tilde{\Gamma} ^q _{nj}\tilde{R}_{qjpk}\\
&+\tilde{\Gamma} ^p _{m0} \tilde{\Gamma} ^q _{nj}\tilde{R}_{iqpk}
+\tilde{\Gamma} ^p _{m0} \tilde{\Gamma} ^q _{lp}\tilde{R}_{ijqk}+\tilde{\Gamma} ^p _{m0} \tilde{\Gamma} ^q _{nk}\tilde{R}_{ijpq}
+\tilde{\Gamma} ^p _{mk} \tilde{\Gamma} ^q _{n0}\tilde{R}_{ijqp}\Bigl] \\
&-\tilde{g}^{\alpha \beta}\Bigl[\tilde{\Gamma} ^0 _{\alpha \beta} \tilde{\nabla}_{0} \tilde{R}_{ij0k}-\tilde{\Gamma}^{\gamma} _{\alpha 0} \tilde{\Gamma}^0 _{\beta \gamma} \tilde{R_{ij0k}} \Bigl]\\
\underset{N}{=}&g^{mn}[\nabla _m \nabla _n P_{ijk}+(\nabla _m R^p _n )R_{ijpk}+R^p _n\nabla _m R_{ijpk}-R^p _m\nabla _n R_{ijpk}]\\
&-\tilde{\nabla} _{0} \tilde{R}_{ij0k}+\frac{1}{2t}P_{ijk}\, , \\
\end{align*}
\begin{align*}
\tilde{\Delta}\tilde{R}_{i0j0}=&g^{kl}[\tilde{\nabla}_k \tilde{\nabla}_l \tilde{R}_{i0j0}]+\tilde{g}^{\alpha \beta}[\tilde{\nabla}_{\alpha} \tilde{\nabla}_{\beta} \tilde{R}_{i0j0}]\\
=&g^{kl}\Bigl[\nabla _k \nabla _l \tilde{R}_{i0j0}-\frac{\partial \tilde{\Gamma ^p _{l0}}}{\partial x^k}\tilde{R}_{ipj0}-\frac{\partial \tilde{\Gamma ^p _{l0}}}{\partial x^k}\tilde{R}_{i0jp}\\
&-\tilde{\Gamma} ^p _{l0}\frac{\partial }{\partial x^k}\tilde{R}_{ipj0}-\tilde{\Gamma} ^p _{l0}\frac{\partial }{\partial x^k}\tilde{R}_{i0jp}
-\tilde{\Gamma} ^p _{k0}\frac{\partial }{\partial x^l}\tilde{R}_{ipj0}-\tilde{\Gamma} ^p _{k0}\frac{\partial }{\partial x^l}\tilde{R}_{i0jp}\\
&+\tilde{\Gamma} ^p _{kl} \tilde{\Gamma} ^q _{p0}\tilde{R}_{iqj0}+\tilde{\Gamma} ^p _{kl} \tilde{\Gamma} ^q _{p0}\tilde{R}_{i0jq}
+\tilde{\Gamma} ^p _{ki} \tilde{\Gamma} ^q _{l0}\tilde{R}_{pqj0}+\tilde{\Gamma} ^p _{ki} \tilde{\Gamma} ^q _{l0}\tilde{R}_{p0jq}\\
&+\tilde{\Gamma} ^p _{k0} \tilde{\Gamma} ^q _{li}\tilde{R}_{qpj0}+\tilde{\Gamma} ^p _{k0} \tilde{\Gamma} ^q _{lp}\tilde{R}_{iqj0}
+\tilde{\Gamma} ^p _{k0} \tilde{\Gamma} ^q _{lj}\tilde{R}_{ipq0}+\tilde{\Gamma} ^p _{k0} \tilde{\Gamma} ^q _{l0}\tilde{R}_{ipjq}\\
&+\tilde{\Gamma} ^p _{kj} \tilde{\Gamma} ^q _{l0}\tilde{R}_{iqp0}+\tilde{\Gamma} ^p _{kj} \tilde{\Gamma} ^q _{l0}\tilde{R}_{i0jq}
+\tilde{\Gamma} ^p _{k0} \tilde{\Gamma} ^q _{li}\tilde{R}_{q0jp}+\tilde{\Gamma} ^p _{k0} \tilde{\Gamma} ^q _{l0}\tilde{R}_{iqjp}\\
&+\tilde{\Gamma} ^p _{k0} \tilde{\Gamma} ^q _{lj}\tilde{R}_{i0qp}+\tilde{\Gamma} ^p _{k0} \tilde{\Gamma} ^q _{lp}\tilde{R}_{i0jq}\Bigl]\\
&-\tilde{g}^{\alpha \beta}\Bigl[ \tilde{\Gamma}^{0}_{\alpha \beta}\tilde{\nabla}_{0}\tilde{R}_{i0j0}-2\tilde{\Gamma} ^{\gamma} _{\alpha 0} \tilde{\Gamma} ^0 _{\beta \gamma}\tilde{R}_{i0j0}\\
&-\tilde{\Gamma} ^{\gamma} _{\alpha 0} \tilde{\Gamma} ^{\delta} _{\beta 0}\tilde{R}_{i\gamma j \delta}
-\tilde{\Gamma} ^{\gamma} _{\alpha 0} \tilde{\Gamma} ^{\delta} _{\beta 0}\tilde{R}_{i\delta j \gamma} \Bigl]\\
\underset{N}{=}&g^{kl}[\nabla _k \nabla _l M_{ij}+\nabla _k R^p _l P_{ipj}+\nabla _k R^p _l P_{jpi}\\
&-R^p _l\nabla _k P_{ipj}-R^p _l\nabla _k P_{jpi}-R^p _k\nabla _l P_{ipj}-R^p _k\nabla _l P_{jpi}\\
&+\tilde{\Gamma} ^p _{k0} \bar{\Gamma} ^q _{l0}\tilde{R}_{ipjq}+\tilde{\Gamma} ^p _{k0} \tilde{\Gamma} ^q _{l0}\tilde{R}_{iqjp}]-\tilde{\nabla}_{0}\tilde{R}_{i0j0}+\frac{1}{t} M_{ij}-\frac{1}{2t^2}R_{ij}\\
\underset{N}{=}&\Delta M_{ij}-\tilde{\nabla}_{0}\tilde{R}_{i0j0}+\frac{1}{2}\nabla ^p \R [P_{pij}+P_{pji}]\\
&+2R^{pq}[\nabla _q P_{pij}+\nabla _q P_{pji}]+2R^p_k R^{kq}R_{iqjp}+\frac{1}{t}M_{ij}-\frac{1}{2t^2}R_{ij} \, ,
\end{align*}
where \begin{equation*}
\tilde{\Delta} \tilde{R}_{abcd}:=(\tilde{\Delta} \widetilde{\Rm } )\Bigl( \frac{\partial}{\partial x^a}, \frac{\partial}{\partial x^b}, \frac{\partial}{\partial x^c}, \frac{\partial}{\partial x^d} \Bigl) \, .
\end{equation*}
Therefore, we have 
\begin{equation}\label{eq:333}
\begin{aligned}
\tilde{\Delta}\tilde{R}_{ijkl}\underset{N}{=}&\Delta R_{ijkl}-\tilde{\nabla}_{0} \tilde{R}_{ijkl} \, , \\
\tilde{\Delta}\tilde{R}_{ij0k}\underset{N}{=}&\Delta P_{ijk}-\tilde{\nabla}_{0} \tilde{R}_{ij0k} +\frac{1}{2}(\nabla ^m R)R_{ijmk}+2R^{lm}\nabla _l R_{ijmk} +\frac{1}{2t}P_{ijk} \, , \\
\tilde{\Delta}\tilde{R}_{i0j0}\underset{N}{=}&\Delta M_{ij}-\tilde{\nabla}_{0} \tilde{R}_{i0j0}+\frac{1}{2}\nabla ^{m} \R (P_{mij}+P_{mji})\\
&+2R^{mn}(\nabla _n P_{mij}+\nabla _n P_{mji})+2R^m_k R^{kn}R_{injm}+\frac{1}{t}M_{ij} -\frac{1}{2t^2}R_{ij} \, ,
\end{aligned}
\end{equation}
Moreover, we have 
\begin{equation}\label{eq:25}
\begin{aligned}
\tilde{\nabla}_{0}\tilde{R}_{ijkl} \underset{N}{=}&\frac{\partial}{\partial t}R_{ijkl}+R^m_i R_{mjkl} +R^m_j R_{imkl} +R^m_k R_{ijml} +R^m_l R_{ijkm} \, , \\ 
\tilde{\nabla}_{0}\tilde{R}_{ij0k} \underset{N}{=}&\frac{\partial}{\partial t}P_{ijk}+R^m_i P_{mjk} +R^m_j P_{imk}\\
&+R^m_k P_{ijm} +\frac{1}{2}(\nabla ^m R)R_{ijmk}+\frac{1}{2t}P_{ijk} \, , \\
\tilde{\nabla}_{0}\tilde{R}_{i0j0} \underset{N}{=}&\frac{\partial}{\partial t}M_{ij}+R^m_i M_{mj} +R^m_j M_{im}+\frac{1}{2}\nabla ^m \R (P_{mij}+P_{mji})+\frac{1}{t}M_{ij}\, ,
\end{aligned}
\end{equation}
where 
\begin{equation*}
\tilde{\nabla}_0 \tilde{R}_{abcd}:=(\tilde{\nabla} _{\frac{\partial}{\partial t}}\widetilde{\Rm} )(\frac{\partial}{\partial x^a}, \frac{\partial}{\partial x^b}, \frac{\partial}{\partial x^c}, \frac{\partial}{\partial x^d} ) \, . 
\end{equation*}
Indeed, by choosing an orthonormal frames on $M\times \mathbb{H}^N$ at a point, we compute  
\begin{align*}
\tilde{\nabla}_{0}\tilde{R}_{ijkl}=&\frac{\partial}{\partial t}\tilde{R}_{ijkl}-\tilde{\Gamma}^m_{0i} \tilde{R}_{mjkl}-\tilde{\Gamma}^m_{0j}\tilde{R}_{imkl}-\tilde{\Gamma}^m_{0k} \tilde{R}_{ijml}-\tilde{\Gamma}^m_{0l} \tilde{R}_{ijkm}\\
\underset{N}{=}&\frac{\partial}{\partial t}R_{ijkl}+R^m_i R_{mjkl} +R^m_j R_{imkl} +R^m_k R_{ijml} +R^m_l R_{ijkm} \, , \\ 
\tilde{\nabla}_{0}\tilde{R}_{ij0k}=&\frac{\partial}{\partial t}\tilde{R}_{ij0k}-\tilde{\Gamma}^m_{0i} \tilde{R}_{mj0k}-\tilde{\Gamma}^m_{0j}\tilde{R}_{im0k}-\tilde{\Gamma}^a_{00} \tilde{R}_{ijak}-\tilde{\Gamma}^m_{0k} \tilde{R}_{ij0m}\\
\underset{N}{=}&\frac{\partial}{\partial t}P_{ijk}+R^m_i P_{mjk} +R^m_j P_{imk} +\frac{1}{2}(\nabla ^m \R )R_{ijmk}\\
&+\frac{1}{2t}P_{ijk}+R^m_k P_{ijm}\, , \\
\tilde{\nabla}_{0}\tilde{R}_{i0j0}=&\frac{\partial}{\partial t}\tilde{R}_{i0j0}-\tilde{\Gamma}^m_{0i} \tilde{R}_{m0j0}-\tilde{\Gamma}^a_{00}\tilde{R}_{iaj0}-\tilde{\Gamma}^m_{0j} \tilde{R}_{i0m0}-\tilde{\Gamma}^a_{00} \tilde{R}_{i0ja}\\
\underset{N}{=}&\frac{\partial}{\partial t}M_{ij}+R^m_i M_{mj} +R^m_j M_{im}+\frac{1}{2}\nabla ^m \R (P_{mij}+P_{mji})+\frac{1}{t}M_{ij}\, .
\end{align*}
From the equations $(\ref{eq:333})$, $(\ref{eq:25})$, $(\ref{e:7})$, and Lemma 2.1, we have
\begin{equation}\label{eq:3333}
\begin{aligned}
\tilde{\Delta}\tilde{R}_{ijkl}\underset{N}{=}&-2R_{imjn}R^{\,\,\, m\,\, n}_{k \,\,\,\, l}+2R_{imjn}R^{\,\,\, m\,\, n}_{l \,\,\,\, k}-2R_{imkn}R^{\,\,\, m\,\, n}_{j \,\,\,\, l}+2R_{imln}R^{\,\,\, m\,\, n}_{j \,\,\,\, k} \, .\\
\tilde{\Delta}\tilde{R}_{ij0k}\underset{N}{=}&-2R_{imjn}P^{mn} _{\,\,\,\,\,\, k} -2R_{jmkn}P_{i}^{\,\,\, mn}-2R_{imkn}P_{j}^{\,\,\, mn} \, , \\
\tilde{\Delta}\tilde{R}_{i0j0}\underset{N}{=}&-2R_{imjn}M^{mn}+2P_{imn} P_{j}^{\,\,\, mn}+4P_{imn}P_{j}^{\,\, nm} \, .
\end{aligned}
\end{equation}

In fact, we can deduce these equations by using another method, which is not rigorous. 
If a Riemannian manifold is Ricci flat, the evolution equation $(\ref{e:7})$ can be described as 
\begin{equation*}
\begin{aligned}
\Delta R_{ijkl}=-2R_{imjn}R^{\,\,\, m\,\, n}_{k \,\,\,\, l}+2R_{imjn}R^{\,\,\, m\,\, n}_{l \,\,\,\, k}-2R_{imkn}R^{\,\,\, m\,\, n}_{j \,\,\,\, l}+2R_{imln}R^{\,\,\, m\,\, n}_{j \,\,\,\, k} \, .
\end{aligned}
\end{equation*}
When we formally apply this equation to the components of the full curvature tensor $\tilde{R}_{abcd}$ , we get the equation which is equivalent to $(\ref{eq:3333})$. \\
\\
We consider the curvature operator as a section of $Sym^2(\wedge ^2 T^{\ast}\tilde{M})$ . 
The curvature operator is a symmetric operator on the space of 2-forms $\tilde{U}^{ab}$ on $\tilde{M}$ defined by 
$$\widetilde{\Rm }(\tilde{U}, \tilde{U})=\tilde{R}_{abcd}\tilde{U}^{ab}\tilde{U}^{cd} \, .$$
Then, the Harnack expression appears as a part of $\widetilde{\Rm }(\tilde{U}, \tilde{U})$ 
in the sense that all terms other than those involved in the Harnack expression is of magnitude $O(\frac{1}{N})$ as $N \rightarrow \infty$. 
However, if $(\tilde{U}^{ab})$'s with at least one index from the $\mathbb{H}^N$ factor are chosen independent of $N$, then $\widetilde{\Rm}(\tilde{U}, \tilde{U})$ 
diverges as $N \rightarrow \infty$ since the dimension of $\tilde{M}^{n+N+1}$ goes to $\infty$. 
Therefore, we are forced to consider the restriction to $\bar{M}=M\times \mathbb{R}^{+}$ which we will discuss in the next section. \\
\\
\subsection{The differential equations for the coefficients of the curvature tensors on the space-time}
As is observed at the end of the previous section, the quantity $\widetilde{\Rm}(\tilde{U}, \tilde{U})$ diverges as $N \rightarrow \infty$ 
unless we do not introduce $N$-dependence in the part of $(\tilde{U}^{ab})$'s which include at least one index from the $\mathbb{H}^N$-part. 
Therefore, we should introduce an $N$-dependence on such $\tilde{U}^{ab}$'s so that the contribution from this part becomes negligible as $N \rightarrow \infty$. 
The simplest way in doing so is that we choose to work on the restriction to the slice $\bar{M}:=M\times \mathbb{R}^{+}$ defined by neglecting $\mathbb{H}^N$-component. 
We define the metric $\bar{g}$ on $\bar{M}$ as follows: 
\begin{align*}
\bar{g}_{00}=\R -\frac{N}{2t} \, ,\quad \bar{g}_{ij}=g_{ij} \, ,\quad \bar{g}_{0i}=0 \, , 
\end{align*}
where $i, j$ are coordinate indices on the $M$ factor, and 0 represent the index of the time coordinate $t$. 
Then, $(\bar{M}, \bar{g})$ is not always Ricci flat up to errors of order $\frac{1}{N}$. 
However, The full curvature tensor of the metric $\bar{g}$ gives Hamilton's Harnack expression. 
In the same way in section 3.2, we see that how the curvature tensor $\bar{R}_{abcd}$ evolves in the direction $\mathbb{R}^+$. 
Note that we compute at a point $x^i \in M $ . 
Then, we get the following equations:
\begin{equation}\label{eq:2}
\begin{aligned}
\bar{\nabla}_{0}\bar{R}_{ijkl} \underset{N}{=}&\frac{\partial}{\partial t}R_{ijkl}+R^m_i R_{mjkl} +R^m_j R_{imkl} +R^m_k R_{ijml} +R^m_l R_{ijkm}\, , \\ 
\bar{\nabla}_{0}\bar{R}_{ij0k} \underset{N}{=}&\frac{\partial}{\partial t}P_{ijk}+R^m_i P_{mjk} +R^m_j P_{imk} +R^m_k P_{ijm} \\
&+\frac{1}{2}(\nabla ^m \R )R_{ijmk}+\frac{1}{2t}P_{ijk}\, , \\
\bar{\nabla}_{0}\bar{R}_{i0j0} \underset{N}{=}&\frac{\partial}{\partial t}M_{ij}+R^m_i M_{mj} +R^m_j M_{im}+\frac{1}{2}\nabla ^m \R (P_{mij}+P_{mji})+\frac{1}{t}M_{ij}\, ,
\end{aligned}
\end{equation}
where 
\begin{equation*}
\bar{\nabla}_0 \bar{R}_{abcd}:=(\bar{\nabla} _{\frac{\partial}{\partial t}}\overline{\Rm} )\Bigl( \frac{\partial}{\partial x^a}, \frac{\partial}{\partial x^b}, \frac{\partial}{\partial x^c}, \frac{\partial}{\partial x^d}\Bigl) \, .
\end{equation*} 
Moreover, computing the Laplacian of $\bar{R}_{abcd}$ by using the formula $(\ref{eq:121})$, we have 
\begin{equation}\label{eq:3}
\begin{aligned}
\bar{\Delta}\bar{R}_{ijkl} \underset{N}{=}&\Delta R_{ijkl}\, , \\
\bar{\Delta}\bar{R}_{ij0k} \underset{N}{=}&\Delta P_{ijk}+\frac{1}{2}(\nabla ^m \R )R_{ijmk}+2R^{lm}\nabla _l R_{ijmk}\, , \\
\bar{\Delta}\bar{R}_{i0j0} \underset{N}{=}&\Delta M_{ij}+\frac{1}{2}\nabla ^m \R [P_{mij}+P_{mji}]+2R^{mn}[\nabla _n P_{mij}+\nabla _n P_{mji}] \\
&+2R^m_k R^{kn}R_{injm}\, ,
\end{aligned}
\end{equation}
where 
\begin{equation*}
\bar{\Delta} \bar{R}_{abcd}:=(\bar{\Delta} \overline{\Rm} )\Bigl( \frac{\partial}{\partial x^a}, \frac{\partial}{\partial x^b}, \frac{\partial}{\partial x^c}, \frac{\partial}{\partial x^d} \Bigl) \, .
\end{equation*}
Indeed, 
\begin{align*}
\bar{\Delta}\bar{R}_{ij0k}=&\bar{g}^{ab}[\bar{\nabla}_a \bar{\nabla}_b \bar{R}_{ij0k}]\\
 \underset{N}{=}&g^{mn}[\bar{\nabla}_m \bar{\nabla}_n \bar{R}_{ij0k}]\\
 \underset{N}{=}&g^{mn}\Bigl[ \nabla _m \nabla _n \bar{R}_{ij0k}-\frac{\partial \bar{\Gamma ^p _{n0}}}{\partial x^m}\bar{R}_{ijpk}
-\bar{\Gamma} ^p _{n0}\frac{\partial }{\partial x^m}\bar{R}_{ij0k}
-\bar{\Gamma} ^p _{m0}\frac{\partial }{\partial x^n}\bar{R}_{ijpk}\\
&+\bar{\Gamma} ^p _{mn} \bar{\Gamma} ^q _{p0}\bar{R}_{ijqk}
+\bar{\Gamma} ^p _{mi} \bar{\Gamma} ^q _{n0}\bar{R}_{pjqk}
+\bar{\Gamma} ^p _{mj} \bar{\Gamma} ^q _{n0}\bar{R}_{ipqk}
+\bar{\Gamma} ^p _{m0} \bar{\Gamma} ^q _{nj}\bar{R}_{qjpk}\\
&+\bar{\Gamma} ^p _{m0} \bar{\Gamma} ^q _{nj}\bar{R}_{iqpk}
+\bar{\Gamma} ^p _{m0} \bar{\Gamma} ^q _{lp}\bar{R}_{ijqk}+\bar{\Gamma} ^p _{m0} \bar{\Gamma} ^q _{nk}\bar{R}_{ijpq}
+\bar{\Gamma} ^p _{mk} \bar{\Gamma} ^q _{n0}\bar{R}_{ijqp}\Bigl] \\
 \underset{N}{=}&g^{mn}[\nabla _m \nabla _n P_{ijk}+(\nabla _m R^p _n )R_{ijpk}+R^p _n\nabla _m R_{ijpk}-R^p _m\nabla _n R_{ijpk}]\\
 \underset{N}{=}&\Delta P_{ijk}+\frac{1}{2}(\nabla ^m \R )R_{ijmk}+2R^{lm}\nabla _l R_{ijmk} \, ,
\end{align*}
\begin{align*}
\bar{\Delta}\bar{R}_{i0j0} \underset{N}{=}&g^{kl}[\bar{\nabla}_k \bar{\nabla}_l \bar{R}_{i0j0}]\\
 \underset{N}{=}&g^{kl}\Bigl[\nabla _k \nabla _l \bar{R}_{i0j0}-\frac{\partial \bar{\Gamma ^p _{l0}}}{\partial x^k}\bar{R}_{ipj0}-\frac{\partial \bar{\Gamma ^p _{l0}}}{\partial x^k}\bar{R}_{i0jp}\\
&-\bar{\Gamma} ^p _{l0}\frac{\partial }{\partial x^k}\bar{R}_{ipj0}-\bar{\Gamma} ^p _{l0}\frac{\partial }{\partial x^k}\bar{R}_{i0jp}
-\bar{\Gamma} ^p _{k0}\frac{\partial }{\partial x^l}\bar{R}_{ipj0}-\bar{\Gamma} ^p _{k0}\frac{\partial }{\partial x^l}\bar{R}_{i0jp}\\
&+\bar{\Gamma} ^p _{kl} \bar{\Gamma} ^q _{p0}\bar{R}_{iqj0}+\bar{\Gamma} ^p _{kl} \bar{\Gamma} ^q _{p0}\bar{R}_{i0jq}
+\bar{\Gamma} ^p _{ki} \bar{\Gamma} ^q _{l0}\bar{R}_{pqj0}+\bar{\Gamma} ^p _{ki} \bar{\Gamma} ^q _{l0}\bar{R}_{p0jq}\\
&+\bar{\Gamma} ^p _{k0} \bar{\Gamma} ^q _{li}\bar{R}_{qpj0}+\bar{\Gamma} ^p _{k0} \bar{\Gamma} ^q _{lp}\bar{R}_{iqj0}
+\bar{\Gamma} ^p _{k0} \bar{\Gamma} ^q _{lj}\bar{R}_{ipq0}+\bar{\Gamma} ^p _{k0} \bar{\Gamma} ^q _{l0}\bar{R}_{ipjq}\\
&+\bar{\Gamma} ^p _{kj} \bar{\Gamma} ^q _{l0}\bar{R}_{iqp0}+\bar{\Gamma} ^p _{kj} \bar{\Gamma} ^q _{l0}\bar{R}_{i0jq}
+\bar{\Gamma} ^p _{k0} \bar{\Gamma} ^q _{li}\bar{R}_{q0jp}+\bar{\Gamma} ^p _{k0} \bar{\Gamma} ^q _{l0}\bar{R}_{iqjp}\\
&+\bar{\Gamma} ^p _{k0} \bar{\Gamma} ^q _{lj}\bar{R}_{i0qp}+\bar{\Gamma} ^p _{k0} \bar{\Gamma} ^q _{lp}\bar{R}_{i0jq}\Bigl]\\
 \underset{N}{=}&g^{kl}[\nabla _k \nabla _l M_{ij}+\nabla _k R^p _l P_{ipj}+\nabla _k R^p _l P_{jpi}\\
&-R^p _l\nabla _k P_{ipj}-R^p _l\nabla _k P_{jpi}-R^p _k\nabla _l P_{ipj}-R^p _k\nabla _l P_{jpi}\\
&+\bar{\Gamma} ^p _{k0} \bar{\Gamma} ^q _{l0}\bar{R}_{ipjq}+\bar{\Gamma} ^p _{k0} \bar{\Gamma} ^q _{l0}\bar{R}_{iqjp}]\\
 \underset{N}{=}&\Delta M_{ij}+\frac{1}{2}\nabla ^p \R [P_{pij}+P_{pji}]\\
&+2R^{pq}[\nabla _q P_{pij}+\nabla _q P_{pji}]+2R^p_k R^{kq}R_{iqjp} \, .
\end{align*}
Here, we have used the contracted second Bianchi identity 
$(\ref{eq:1})$ again. 
Moreover, the difference between the equation $(\ref{eq:105})$ and $(\ref{eq:106})$ is whether the derivative of the curvature tensor in the direction $\mathbb{R}^+$ appears or not, 
so that the equation 
\begin{equation}\label{eq:106}
\bar{\Delta} \bar{R}_{ijkl} \underset{N}{=}\Delta R_{ijkl}
\end{equation}
holds (the $\tilde{\nabla}_0$-covariant derivative stems from the covariant derivative in the $\mathbb{H}^N$ direction via the formula $(\ref{eq:121})$). \\
\\
From the equations $(\ref{eq:2})$ and $(\ref{eq:3})$, we have 
\begin{equation*}
\begin{aligned}
(\bar{\nabla}_0 - \bar{\Delta} )\bar{R}_{ijkl}\underset{N}{=}&\Bigl( \frac{\partial}{\partial t}-\Delta \Bigl) R_{ijkl}\\
&+R^m_i R_{mjkl} +R^m_j R_{imkl} +R^m_k R_{ijml} +R^m_l R_{ijkm}\, , \\
(\bar{\nabla}_0- \bar{\Delta} )\bar{R}_{ij0k}\underset{N}{=}&\Bigl( \frac{\partial}{\partial t}-\Delta \Bigl) P_{ijk}+R^m_i P_{mjk} +R^m_j P_{imk} +R^m_k P_{ijm}\\
&+\frac{1}{2t}P_{ijk}-2R^{lm}\nabla _l R_{ijmk}\, , \\
(\bar{\nabla}_0- \bar{\Delta} )\bar{R}_{i0j0}\underset{N}{=}&\Bigl( \frac{\partial}{\partial t}-\Delta \Bigl) M_{ij}+R^m_i M_{mj} +R^m_j M_{im}\\
&+\frac{1}{t}M_{ij}-2R^{mn}(\nabla _n P_{mij}-\nabla _n P_{mji})-2R^m_k R^{kn}R_{injm} \, .
\end{aligned}
\end{equation*}
From (\ref{e:7}) and (\ref{eq:21}), we have the following proposition:
\begin{proposition}
\begin{equation}\label{eq:9}
\begin{aligned}
(\bar{\nabla}_0 - \bar{\Delta} )\bar{R}_{ijkl}\underset{N}{=}&2R_{imjn}R^{\,\, m\,\, n}_{k \,\,\,\, l}-2R_{imjn}R^{\,\, m\,\, n}_{l \,\,\,\, k}\\
&+2R_{imkn}R^{\,\, m\,\, n}_{j \,\,\,\, l}-2R_{imln}R^{\,\, m\,\, n}_{j \,\,\,\, k}\, , \\
(\bar{\nabla}_0- \bar{\Delta} )\bar{R}_{ij0k}\underset{N}{=}&2R_{imjn}P^{mn}_{\,\,\,\,\,\,\,\,\, k}+2R_{imkn}P^{m\,\,\, n}_{\,\,\,\,\,\, j}+2R_{imkn}P^{\,\, m n}_{i}+\frac{1}{2t}P_{ijk}\, , \\
(\bar{\nabla}_0- \bar{\Delta} )\bar{R}_{i0j0}\underset{N}{=}&2R_{imjn}M^{mn}+2P_{imn}P^{\,\, m n}_{j}\\
&-4P_{imn}P^{\,\, n m}_{j}-\frac{1}{2t^2}R_{ij}+\frac{1}{t}M_{ij}\, .
\end{aligned}
\end{equation}
\end{proposition}

\section{Proof of Hamilton's Harnack inequality}
\subsection{The key Lemma of the proof of Main Theorem}

Recall that for any 2-form $\bar{U}^{ab}$ on $\bar{M}$, we can write
\begin{equation*}
\overline{\Rm}(\bar{U}, \bar{U})=\bar{R}_{abcd} \bar{U}^{ab} \bar{U}^{cd} \, .
\end{equation*}
In section 4, we prove that the curvature tensor $\overline{\Rm}(\bar{U}, \bar{U})$ is weakly positive when the curvature operator $R_{ijkl}U^{ij} U^{kl}$ on $M$ is weakly positive. 
First, we can pick a 2-form $\bar{U}^{ab}$ such that 

\begin{equation*}
\left \{
\begin{alignedat}{1}
&\bar{\nabla} _i\bar{U}^{0j}=0\\
& (\bar{\nabla} _0-\bar{\Delta})\bar{U}^{ij}=0
\end{alignedat}
\right .
\end{equation*}
at a point on $\bar{M}$. Then, we have

\begin{equation}\label{eq:41}
\begin{aligned}
(\bar{\nabla}_0 -\bar{\Delta})\overline{\Rm}(\bar{U}, \bar{U})\underset{N}{=}&((\bar{\nabla}_{0}-\bar{\Delta} )\bar{R}_{ijkl})\bar{U}^{ij}\bar{U}^{kl}+2((\bar{\nabla}_{0}-\bar{\Delta})\bar{R}_{ij0k})\bar{U}^{ij}\bar{U}^{0k}\\
&+((\bar{\nabla}_{0}-\bar{\Delta})\bar{R}_{0i0j})\bar{U}^{0i}\bar{U}^{0j}\\
&-4(\bar{\nabla}^p \bar{R}_{ijkl})(\bar{\nabla}_p \bar{U}^{ij})\bar{U}^{kl}-2\bar{R}_{ijkl}(\bar{\nabla}_p \bar{U}^{ij})(\bar{\nabla}^p \bar{U}^{kl})\\
&-4(\bar{\nabla}^p \bar{R}_{ij0k})(\bar{\nabla}_p \bar{U}^{ij})\bar{U}^{0k}\\
&+2\bar{R}_{ij0k}\bar{U}^{ij}(\bar{\nabla}_0 -\bar{\Delta})\bar{U}^{0k}+2\bar{R}_{0i0j}(\bar{\nabla}_0 -\bar{\Delta})\bar{U}^{0i}\bar{U}^{0j}\, ,
\end{aligned}
\end{equation}
where $(\bar{\nabla}_0 -\bar{\Delta})\overline{\Rm}(\bar{U}, \bar{U})=\bar{\nabla}_{0}\Bigl( \overline{\Rm}(\bar{U}, \bar{U}) \Bigl)- \bar{\Delta}\Bigl( \overline{\Rm}(\bar{U}, \bar{U}) \Bigl)$. \\
\\
From the equations $(\ref{eq:9})$, we have
\begin{equation}\label{eq:8}
\begin{aligned}
((\bar{\nabla}_0 &- \bar{\Delta} )\bar{R}_{ijkl})\bar{U}^{ij}\bar{U}^{kl}\\
&\underset{N}{=}(2R_{imjn}R^{\,\, m\,\, n}_{k \,\,\,\, l}-2R_{imjn}R^{\,\, m\,\, n}_{l \,\,\,\, k}+2R_{imkn}R^{\,\, m\,\, n}_{j \,\,\,\, l}-2R_{imln}R^{\,\, m\,\, n}_{j \,\,\,\, k})\bar{U}^{ij}\bar{U}^{kl}\\
&\underset{N}{=}R_{ijmn} R_{kl}^{\,\,\,\,\, mn} \bar{U}^{ij}\bar{U}^{kl}+4R_{imkn}R^{\,\, m\,\, n}_{j \,\,\,\, l}\bar{U}^{ij}\bar{U}^{kl}\\
((\bar{\nabla}_0&- \bar{\Delta} )\bar{R}_{ij0k})\bar{U}^{ij}\bar{U}^{0k}\\
&\underset{N}{=}(2R_{imjn}P^{mn}_{\,\,\,\,\,\,\,\,\, k}+2R_{imkn}P^{m\,\,\, n}_{\,\,\,\,\,\, j}+2R_{imkn}P^{\,\, m n}_{j}+\frac{1}{2t}P_{ijk})\bar{U}^{ij}\bar{U}^{0k}\\
&\underset{N}{=}(2R_{imjn}P^{mn}_{\,\,\,\,\,\,\,\,\, k}+4R_{imkn}P^{m\,\,\, n}_{\,\,\,\,\,\, j}+\frac{1}{2t}P_{ijk})\bar{U}^{ij}\bar{U}^{0k}\\
((\bar{\nabla}_0&- \bar{\Delta} )\bar{R}_{0i0j})\bar{U}^{0i}\bar{U}^{0j}\\
&\underset{N}{=}(2R_{imjn}M^{mn}+2P_{imn}P^{\,\, m n}_{j}-4P_{imn}P^{\,\, n m}_{j}-\frac{1}{2t^2}R_{ij}+\frac{1}{t}M_{ij})\bar{U}^{0i}\bar{U}^{0j}\\
&\underset{N}{=}(2R_{imjn}M^{mn}+P_{mni}P^{m n}_{\,\,\,\,\,\,\,\,\,j}-2P_{imn}P^{\,\, n m}_{j}-\frac{1}{2t^2}R_{ij}+\frac{1}{t}M_{ij})\bar{U}^{0i}\bar{U}^{0j}.
\end{aligned}
\end{equation}
Indeed, we can see that the following equations hold:
\begin{equation*}
\begin{aligned}
2R_{imjn}R^{\,\, m\,\, n}_{k \,\,\,\, l}&-2R_{imjn}R^{\,\, m\,\, n}_{l \,\,\,\, k}\\
&=R_{imjn}R^{\,\, m\,\, n}_{k \,\,\,\, l}+R_{jmin}R^{\,\, m\,\, n}_{l \,\,\,\, k}-R_{imjn}R^{\,\, m\,\, n}_{l \,\,\,\, k}-R_{jmin}R^{\,\, m\,\, n}_{k \,\,\,\, l}\\
&=(R_{imjn}-R_{jmin})(R^{\,\, m\,\, n}_{k \,\,\,\, l}-R^{\,\, m\,\, n}_{l \,\,\,\, k})\\
&=R_{ijmn}R_{kl}^{\,\,\,\,\,\,mn} \, ,
\end{aligned}
\end{equation*}
\begin{equation*}
\begin{aligned}
2P_{imn}P^{\,\, m n}_{j}-2P_{imn}P^{\,\, n m}_{j}&=P_{imn}P_{j}^{\,\, mn}+P_{min}P_{\,\,\,\, j}^{m\,\, n}-P_{imn}P_{j}^{\,\, nm}-P_{min}P_{\,\,\,\, j}^{n\,\, m}\\
&=(P_{nim}+P_{imn})(P_{\,\,\,\,j}^{n\,\, m}+P_{j}^{\,\, mn})\\
&=P_{mni}P^{m n}_{\,\,\,\,\,\,\,\,\,j} \, .
\end{aligned}
\end{equation*}
Here we have used the formula $P_{ijk}+P_{jki}+P_{kij}=0$. 
Then, we have the following lemma: 
\begin{lemma}\label{l:3}
If a 2-form $\bar{U}^{ab}$ satisfies
\begin{equation}\label{as:1}
\left \{
\begin{alignedat}{1}
&\bar{\nabla} _k\bar{U}^{ij}=\frac{1}{4t}(\bar{\delta} ^{i}_{k}\bar{U}^{0j}-\bar{\delta} ^{j}_{k}\bar{U}^{0i})\, , \\
&\bar{\nabla} _i\bar{U}^{0j}=0\, , \\
&(\bar{\nabla} _0 -\bar{\Delta})\bar{U}^{ij}=0\, , \\
&(\bar{\nabla} _0 -\bar{\Delta})\bar{U}^{0i}=\frac{1}{2t}\bar{U}^{0i}\, ,
\end{alignedat}
\right .
\end{equation}
at a point on $\bar{M}$, then we have 
\begin{equation}\label{eq:7}
\begin{aligned}
(\bar{\nabla}_0 -\bar{\Delta})\overline{\Rm}(\bar{U}, \bar{U})\underset{N}{=}&2R_{ikjl}M^{kl}\bar{U}^{0i}\bar{U}^{0j}-2P_{ikl}P_j^{\,\,l k}\bar{U}^{0i}\bar{U}^{0j}\\
&+8R_{ilkm}P^{l \,\, m}_{\,\,j}\bar{U}^{ij}\bar{U}^{0k}+4R_{imkn}R^{\,\, m\,\, n}_{j\,\, l}\bar{U}^{ij}\bar{U}^{kl}\\
&+[P_{ijm}\bar{U}^{0m}+R_{ijmn}\bar{U}^{mn}][P_{kln}\bar{U}^{0n}+R_{klpq}\bar{U}^{pq}]\, ,
\end{aligned}
\end{equation}
where $\bar{\delta} ^{i}_{j}:=\bar{g}^{ia}\bar{g}_{ja}$ . 
\end{lemma}
\begin{proof}
We can pick a 2-form $U^{ab}$ such that
\begin{equation*}
\bar{\nabla} _k\bar{U}^{ij}=\frac{1}{4t}(\bar{\delta} ^{i}_{k}\bar{U}^{0j}-\bar{\delta} ^{j}_{k}\bar{U}^{0i}) \, 
\end{equation*}
holds at a point. By using the formula $\nabla ^p R_{pijk}=P_{jki}$ from the second Bianchi identity $(\ref{eq:22})$, We have
\begin{equation}\label{eq:31}
\begin{aligned}
-4(\bar{\nabla}^p \bar{R}_{ijkl})(\bar{\nabla}_p \bar{U}^{ij})\bar{U}^{kl}
&=-4(\bar{\nabla}^p \bar{R}_{ijkl})\{ \frac{1}{4t}(\delta ^i _p \bar{U}^{0j}-\delta ^j _p \bar{U}^{0i})\} \bar{U}^{kl}\\
&=-\frac{1}{t}\nabla ^p R_{pjkl} \bar{U}^{0j}\bar{U}^{kl}+\frac{1}{t}\nabla ^p R_{ipkl}\bar{U}^{0j}\bar{U}^{kl}\\
&=-\frac{1}{t}P_{klj}\bar{U}^{0j}\bar{U}^{kl}+\frac{1}{t}P_{kli}\bar{U}^{0i}\bar{U}^{kl}\\
&=-\frac{2}{t}P_{ijk}\bar{U}^{ij}\bar{U}^{0k}\\
\end{aligned}
\end{equation}
We can compute all other terms in the equation $(\ref{eq:41})$ in a similar way:
\begin{equation}\label{eq:32}
\begin{aligned}
-2\bar{R}_{ijkl}&(\bar{\nabla}_p \bar{U}^{ij})(\bar{\nabla}^p \bar{U}^{kl})\\
&=-2R_{ijkl}\{ \frac{1}{4t} (\bar{\delta}^{i}_p\bar{U}^{0j}-\bar{\delta}^{j}_{p}\bar{U}^{0i}) \} \{ \frac{1}{4t}(\bar{g}^{kp}\bar{U}^{0l}-\bar{g}^{lp}\bar{U}^{0k}) \}\\
&=-\frac{1}{8t^2}R_{ijkl}\{ g^{ik}\bar{U}^{0j}\bar{U}^{0l}-g^{il}\bar{U}^{0j}\bar{U}^{0k} -g^{jk}\bar{U}^{0i}\bar{U}^{0l} +g^{jl}\bar{U}^{0i}\bar{U}^{0k} \} \\
&=-\frac{1}{8t^2}\{ R_{jl}\bar{U}^{0j}\bar{U}^{0l}+R_{jk}\bar{U}^{0j}\bar{U}^{0k} +R_{il}\bar{U}^{0i}\bar{U}^{0l} +R_{ik}\bar{U}^{0i}\bar{U}^{0k} \} \\
&=-\frac{1}{2t^2}R_{ij}\bar{U}^{0i}\bar{U}^{0j}\\
\end{aligned}
\end{equation}
\begin{equation}\label{eq:33}
\begin{aligned}
-4&(\bar{\nabla}^p \bar{R}_{ij0k})(\bar{\nabla}_p \bar{U}^{ij})\bar{U}^{0k}\\
=&-4\nabla ^p \bar{R}_{ij0k} (\bar{\nabla} _p \bar{U}^{ij})\bar{U}^{0k}+4 \bar{\Gamma}^{m} _{p0}R_{ijmk}(\bar{\nabla}_p \bar{U}^{ij})\bar{U}^{0k}\\
=&-4\nabla ^p P_{ijk}\{ \frac{1}{4t} (\bar{\delta} ^i _p \bar{U}^{0j}-\bar{\delta} ^j _p \bar{U}^{0i})\}\bar{U}^{0k}-4R^m _p R_{ijmk}\{ \frac{1}{4t} (\bar{\delta} ^i _p \bar{U}^{0j}-\bar{\delta} ^j _p \bar{U}^{0i})\} \bar{U}^{0k}\\
=&-\frac{1}{t}\nabla ^{i}P_{ijk}\bar{U}^{0j}\bar{U}^{0k}+\frac{1}{t}\nabla ^{j}P_{ijk}\bar{U}^{0i}\bar{U}^{0k}\\
&+\frac{1}{t}R^{pm}R_{pjmk}\bar{U}^{0j}\bar{U}^{0k}-\frac{1}{t}R^{pm}R_{ipmk}\bar{U}^{0i}\bar{U}^{0k}\\
=&-\frac{2}{t}\nabla ^{p}P_{pij}\bar{U}^{0i}\bar{U}^{0i}-\frac{2}{t}R^{pm}R_{pimj}\bar{U}^{0i}\bar{U}^{0j}. 
\end{aligned}
\end{equation}
From the equation $(\ref{eq:31})$, $(\ref{eq:32})$, $(\ref{eq:33})$, and $(\ref{eq:41})$, we have
\begin{equation*}
\begin{aligned}
(\bar{\nabla}_0 -\bar{\Delta})\overline{\Rm}(\bar{U}, \bar{U})=&((\bar{\nabla}_{0}-\bar{\Delta} )\bar{R}_{ijkl})\bar{U}^{ij}\bar{U}^{kl}+2((\bar{\nabla}_{0}-\bar{\Delta})\bar{R}_{ij0k})\bar{U}^{ij}\bar{U}^{0k}\\
&+((\bar{\nabla}_{0}-\bar{\Delta})\bar{R}_{0i0j})\bar{U}^{0i}\bar{U}^{0j}\\
&-\frac{2}{t}P_{ijk}\bar{U}^{ij}\bar{U}^{0k}-\frac{1}{2t^2}R_{ij}\bar{U}^{0i}\bar{U}^{0j}\\
&-\frac{2}{t}\nabla ^{p}P_{pij}\bar{U}^{0i}\bar{U}^{0i}-\frac{2}{t}R^{pm}R_{pimj}\bar{U}^{0i}\bar{U}^{0j}\\
&+2\bar{R}_{ij0k}\bar{U}^{ij}(\bar{\nabla}_0 -\bar{\Delta})\bar{U}^{ok}+2\bar{R}_{0i0j}(\bar{\nabla}_0 -\bar{\Delta})\bar{U}^{oi}\bar{U}^{0j}. 
\end{aligned}
\end{equation*}
Moreover, if we assume that
\begin{equation*}
(\bar{\nabla} _0 -\bar{\Delta})\bar{U}^{0i}=\frac{1}{2t}\bar{U}^{0i}
\end{equation*}
at a point, then we have 
\begin{equation}\label{eq:34}
\begin{aligned}
2\bar{R}_{ij0k}\bar{U}^{ij}(\bar{\nabla}_0 -\bar{\Delta})\bar{U}^{0k}&=\frac{1}{t}P_{ijk}\bar{U}^{ij}\bar{U}^{0k}\\
2\bar{R}_{0i0j}(\bar{\nabla}_0 -\bar{\Delta})\bar{U}^{0i}\bar{U}^{0j}&=\frac{1}{t}M_{ij}\bar{U}^{0i}\bar{U}^{0j}. 
\end{aligned}
\end{equation}
Hence, we have
\begin{equation}\label{eq:43}
\begin{aligned}
(\bar{\nabla}_0-\bar{\Delta})\overline{\Rm}(\bar{U}, \bar{U})\underset{N}{=}&((\bar{\nabla}_{0}-\bar{\Delta} )\bar{R}_{ijkl})\bar{U}^{ij}\bar{U}^{kl}+2((\bar{\nabla}_{0}-\bar{\Delta})\bar{R}_{ij0k})\bar{U}^{ij}\bar{U}^{0k}\\
&+((\bar{\nabla}_{0}-\bar{\Delta})\bar{R}_{0i0j})\bar{U}^{0i}\bar{U}^{0j}\\
&-\frac{2}{t}P_{ijk}\bar{U}^{ij}\bar{U}^{0k}-\frac{1}{2t^2}R_{ij}\bar{U}^{0i}\bar{U}^{0j}\\
&-\frac{2}{t}\nabla ^{p}P_{pij}\bar{U}^{0i}\bar{U}^{0i}-\frac{2}{t}R^{pm}R_{pimj}\bar{U}^{0i}\bar{U}^{0j}\\
&+\frac{1}{t}P_{ijk}\bar{U}^{ij}\bar{U}^{0k}+\frac{1}{t}M_{ij}\bar{U}^{0i}\bar{U}^{0j}\\
\underset{N}{=}&((\bar{\nabla}_{0}-\bar{\Delta} )\bar{R}_{ijkl})\bar{U}^{ij}\bar{U}^{kl}+2((\bar{\nabla}_{0}-\bar{\Delta})\bar{R}_{ij0k})\bar{U}^{ij}\bar{U}^{0k}\\
&+((\bar{\nabla}_{0}-\bar{\Delta})\bar{R}_{0i0j})\bar{U}^{0i}\bar{U}^{0j}-\frac{1}{t}P_{ijk}\bar{U}^{ij}\bar{U}^{0k}\\
&-(\frac{1}{t}\nabla ^{p}P_{pij}+\frac{1}{t}R^{pm}R_{pimj})\bar{U}^{0i}\bar{U}^{0j}.
\end{aligned}
\end{equation}
Here, we have used the formula 
\begin{equation*}
M_{ij}=\nabla ^p P_{pij}+R_{ikjl}R^{kl}+\frac{1}{2t}R_{ij} \, . 
\end{equation*}
From $(\ref{eq:8})$, the component of $\tilde{U}^{ij}\tilde{U}^{0k}$ in the right hand side of the equation $(\ref{eq:43})$ is expressed as 
\begin{equation*}
\begin{aligned}
&2\Bigl( 2R_{imjn}P^{mn}_{\,\,\,\,\,\,\,\,\, k}+4R_{imkn}P^{m\,\,\, n}_{\,\,\,\,\,\, j}+\frac{1}{2t}P_{ijk}\Bigl) \bar{U}^{ij}\bar{U}^{0k}-\frac{1}{t}P_{ijk}\bar{U}^{ij}\bar{U}^{0k}\\
&=(4R_{imjn}P^{mn}_{\,\,\,\,\,\,\,\,\, k}+8R_{imkn}P^{m\,\,\, n}_{\,\,\,\,\,\, j})\bar{U}^{ij}\bar{U}^{0k}\, .
\end{aligned}
\end{equation*}
Similarly, the component of $\bar{U}^{0i}\bar{U}^{0j}$ becomes
\begin{equation*}
\begin{aligned}
&\Bigl( 2R_{imjn}M^{mn}+P_{mni}P^{m n}_{\,\,\,\,\,\,j}-2P_{imn}P^{\,\, n m}_{j}-\frac{1}{2t^2}R_{ij}+\frac{1}{t}M_{ij}\Bigl) \bar{U}^{0i}\bar{U}^{0j}\\
&-\Bigl( \frac{1}{t}\nabla ^{p}P_{pij}+\frac{1}{t}R^{mn}R_{imjn}\Bigl) \bar{U}^{0i}\bar{U}^{0j}\\
&=(2R_{imjn}M^{mn}+P_{mni}P^{m n}_{\,\,\,\,\,\,j}-2P_{imn}P^{\,\, n m}_{j})\bar{U}^{0i}\bar{U}^{0j}\, .
\end{aligned}
\end{equation*}
Hence, we have
\begin{equation*}
\begin{aligned}
(\bar{\nabla _0}-\bar{\Delta})\overline{\Rm}(\bar{U}, \bar{U})=&(R_{ijmn} R_{kl}^{\,\,\,\,\, mn}+4R_{imkn}R^{\,\, m\,\, n}_{j \,\,\,\, l})\bar{U}^{ij}\bar{U}^{kl}\\
&+(4R_{imjn}P^{mn}_{\,\,\,\,\,\,\,\,\, k}+8R_{imkn}P^{m\,\,\, n}_{\,\,\,\,\,\, j})\bar{U}^{ij}\bar{U}^{0k}\\
&+(2R_{imjn}M^{mn}+P_{mni}P^{m n}_{\,\,\,\,\,j}-2P_{imn}P^{\,\, n m}_{j})\bar{U}^{0i}\bar{U}^{0j}\, .
\end{aligned}
\end{equation*}
Since we can write
\begin{equation*}
\begin{aligned}
&R_{ijmn} R_{kl}^{\,\,\,\,\, mn}\bar{U}^{ij}\bar{U}^{kl}+4R_{imjn}P^{mn}_{\,\,\,\,\,\,\,\,\, k}\bar{U}^{ij}\bar{U}^{0k}+P_{mni}P^{m n}_{\,\,\,\,\,\,j}\bar{U}^{0i}\bar{U}^{0j}\\
=&[P_{ijm}\bar{U}^{0m}+R_{ijmn}\bar{U}^{mn}][P_{kln}\bar{U}^{0n}+R_{klpq}\bar{U}^{pq}] \, ,
\end{aligned}
\end{equation*}
We have thus proved Lemma 4.1.
\end{proof}
\subsection{Comments for Lemma 4.1}
\begin{remark}
{\rm We recall a remark by Hamilton (see Hamilton [11, Lemma 4.5]). 
We see that if $\overline{\Rm}(\bar{U}, \bar{U})$ is weakly positive at a point, then the first and the second terms of the right hand side of the equation $(\ref{eq:7})$ is the sum of squares of linear forms. 
Indeed, if the curvature tensor $\bar{R}_{ijkl}\bar{U}^{ij}\bar{U}^{kl}+2\bar{R}_{ij0k}\bar{U}^{ij}\bar{U}^{0k}+\bar{R}_{0i0j}\bar{U}^{0i}\bar{U}^{0j}$ is weakly positive, 
the tensor is expressed as $[\bar{X}_{ij}\bar{U}^{ij}+\bar{X}_{0k}\bar{U}^{0k}]^2$ by setting $\bar{R}_{abcd}\bar{U}^{ab}\bar{U}^{cd}=(\bar{X}_{ab}\bar{U}^{ab})(\bar{X}_{cd}\bar{U}^{cd})$. 
Then, we can write
\begin{equation}\label{rr:3}
\begin{aligned}
&2R_{ikjl}M^{kl}\bar{U}^{0i}\bar{U}^{0j}-2P_{ikl}P_j^{\,\,l k}\bar{U}^{0i}\bar{U}^{0j}\\
&+8R_{ilkm}P^{l \,\, m}_{\,\,j}\bar{U}^{ij}\bar{U}^{0k}+4R_{imkn}R^{\,\, m\,\, n}_{j\,\,\, l}\bar{U}^{ij}\bar{U}^{kl}\\
=&\sum _{M,N}(\bar{X}_{ik}^M\bar{X}^N_{0k}\bar{U}^{0i}-\bar{X}^N_{jk}\bar{X}^M_{0k}\bar{U}^{0j}-2\bar{X}^M_{ik}\bar{X}^N_{jk}\bar{U}^{ij})^2.
\end{aligned}
\end{equation}
Indeed, we have
\begin{equation*}
\begin{aligned}
&\sum _{M,N}(\bar{X}_{ik}^M\bar{X}^N_{0k}\bar{U}^{0i}-\bar{X}^N_{jk}\bar{X}^M_{0k}\bar{U}^{0j}-2\bar{X}^M_{ik}\bar{X}^N_{jk}\bar{U}^{ij})^2\\
=&(\bar{X}_{im}\bar{X}_{jn})(\bar{X}_{0m}\bar{X}_{0n})\bar{U}^{0i}\bar{U}^{0j}+(\bar{X}_{0m}\bar{X}_{0n})(\bar{X}_{im}\bar{X}_{jn})\bar{U}^{0i}\bar{U}^{0j}\\
&+4(\bar{X}_{im}\bar{X}_{jn})(\bar{X}_{km}\bar{X}_{ln})\bar{U}^{ij}\bar{U}^{kl}-2(\bar{X}_{im}\bar{X}_{0n})(\bar{X}_{0m}\bar{X}_{jn})\bar{U}^{0i}\bar{U}^{0j}\\
&+4(\bar{X}_{0m}\bar{X}_{jn})(\bar{X}_{im}\bar{X}_{kn})\bar{U}^{0k}\bar{U}^{ij}-4(\bar{X}_{im}\bar{X}_{kn})(\bar{X}_{jm}\bar{X}_{0n})\bar{U}^{ij}\bar{U}^{0k}\\
=&2R_{imjn}M^{mn}\bar{U}^{0i}\bar{U}^{0j}+4R_{imkn}R^{\,\, m\,\, n}_{j\,\, l}\bar{U}^{ij}\bar{U}^{kl}\\
&-2P_{imn}P_j^{\,\,n m}\bar{U}^{0i}\bar{U}^{0j}+8R_{ilkm}P_j^{\,\,n m}\bar{U}^{ij}\bar{U}^{0k}
\end{aligned}
\end{equation*}
On the other hand, the third term is clearly a weakly positive quadratic form.} 
\end{remark}
\begin{remark}\label{la:1}
{\rm We note that the equation $(\ref{eq:7})$ looks the same as Hamilton's result [11, Theorem 4.1] if we set 
\begin{equation*}
W^i:=\bar{U}^{0i} \, .
\end{equation*}
However, the way of the extension of the 2-form $U^{ij}$ used by Hamilton is different from ours. 
He gave the geometric interpretation to the way of the extension of the 2-form $U^{ij}$ in the direction $M$ from the following three equations: 
\begin{equation}\label{eq:98}
\left \{
\begin{alignedat}{1}
&U_{ij}=\frac{1}{2}(W_i X_j - W_j X_i) \, ,\\
&\nabla_i X_j=R_{ij}+\frac{1}{2t}g_{ij} \, ,\\
&\nabla_{i} W_{j}=0  \, ,
\end{alignedat}
\right .
\end{equation}
where the second equation implies that the metric $g$ is a gradient expanding soliton. 
When we differentiate the first equation of $(\ref{eq:98})$ in the direction $M$ and substitute the second and third equations of $(\ref{eq:98})$, 
we have 
\begin{equation}\label{eq:99}
{\nabla} _k {U}_{ij}=\frac{1}{2}(R_{ik} W_{j}-R_{jk}W_i)+\frac{1}{4t}(g_{ik}W_j-g_{jk}W_i )
\end{equation}
which is the first equation of the Hamilton's extension: 
\begin{equation}\label{eq:107}
\left \{
\begin{alignedat}{1}
&{\nabla} _k {U}_{ij}=\frac{1}{2}(R_{ik} W_{j}-R_{jk}W_i)+\frac{1}{4t}(g_{ik}W_j-g_{jk}W_i )\, , \\
&{\nabla} _i W_{j}=0\, , \\
&(D_t -\Delta )U_{ij}=0\, , \\
&(D_t-\Delta )W_{i}=\frac{1}{t}W_i  \, .
\end{alignedat}
\right .
\end{equation} 
In this way, Hamilton derived the formulae $(\ref{eq:107})$ by using a gradient expanding soliton as a model, which can be found in the second equation of $(\ref{eq:98})$.
Since we use the hyperbolic thermostat as a model in our setting, our way $(\ref{as:1})$ of extension of the 2-forms $U_{ij}$ is the $R_{ij}=0$ version of $(\ref{eq:107})$. \\
\\
On the other hand, Our extension of $\bar{U}^{0i}$ in the direction of $\mathbb{R}^{+}$ (the forth equation of $(\ref{as:1})$) forces all components of the right hand side of $(\ref{eq:41})$ to be the products in two of $R_{ijkl}$, $P_{ijk}$, $M_{ij}$ 
like $R_{imjn}M^{mn}\bar{U}^{0i}\bar{U}^{0j}$. 
Hence, the way of the extension is determined so that the equation $(\ref{eq:41})$ takes the form like $(\ref{eq:7})$. \\
\\
If we extend 2-form $\bar{U}^{0i}$ as Hamilton did in our setting, 
we see that the right hand side of $(\ref{eq:41})$ becomes the sum of squares as in $(\ref{rr:3})$ and the additional term of the form 
\begin{align}\label{aa:12}
-\frac{1}{t} \left[ P_{ijk} \bar{U}^{ij}\bar{U}^{0k}+M_{ij} \bar{U}^{0i}\bar{U}^{0j} \right] 
\end{align} 
which may be negative. 
We now consider the interpretation of the extra term.  Differentiating the second equation of $(\ref{eq:98})$ gives 
$$\nabla_i \nabla _j X_k -\nabla_j \nabla _i X_k=\nabla_i R_{jk}-\nabla_j R_{ik} \, .$$
On the other hand, we have 
$$\nabla_i \nabla _j X_k -\nabla_j \nabla _i X_k=-R_{ijk} {}^{l} X_{l} \, ,$$
by using Ricci identity $(\ref{re:1})$.  
Hence, we have 
$$\nabla_i R_{jk}-\nabla_j R_{ik}=-R_{ijk} {}^{l}X_{l} \, .$$
We differentiate again and use the second equation of $(\ref{eq:98})$ to get
$$\nabla_i \nabla_j R_{kl}-\nabla _i \nabla_k R_{jl}=-\nabla _i R_{jklm}X^{m}-R_{i}^{m}R_{jklm}-\frac{1}{2t}R_{jkli} \, .$$
If we take the trace of the equation, we have
$$P_{kij}W^{i}W^{j} X^k+M_{ij}W^{i}W^{j}=0$$
for all vector $W^{i}$. 
Hence, by setting $\bar{U}^{ij}=\frac{1}{2}(X^i W^j - X^j W^j)$ and $\bar{U}^{0i}=W^{i}$, 
we have
$$P_{ijk} \bar{U}^{ij} \bar{U}^{0k}+M_{ij} \bar{U}^{0i}\bar{U}^{0j}=0$$ 
which holds under the second equation of $(\ref{eq:98})$, i.e., the term $(\ref{aa:12})$ vanishes under the gradient expanding soliton. }
\end{remark}


\subsection{Proof of Main Theorem}
\subsubsection{The idea of the proof}
We recall Main Theorem of this article, which was stated as Theorem 1.3 in introduction. 
\begin{theorem}[Main Theorem]\label{mt:1}
Let $(M,g_{ij}(t))$ be a complete Ricci flow for $t \in (0, T] \subset \mathbb{R}^{+}$ with uniformly bounded curvature, and 
assume that the manifold $(M, g_{ij}(t))$ has a weakly positive curvature operator. 
Then, the manifold $(\bar{M}, \bar{g}_{ab}(t))$ has a weakly positive curvature operator
\end{theorem}
The proof of Theorem $\ref{mt:1}$ is an application of the maximum principle. 
Just as in Hamilton \cite{H1}, we first give a heuristic proof under stronger assumptions (compactness and strictly positive curvature operator) 
in order to avoid technical complication and to clarify the structure of the proof. 
We will justify the following heuristic argument later. \\
\\
In the setting of Theorem $\ref{mt:1}$, we assume that the manifold $M$ is compact and the Ricci flow $g_{ij}$ has a strictly positive curvature operator. 
Since the manifold $M$ is compact, the tensors $R_{ijkl}$ and $P_{ijk}$ are bounded. 
Furthermore, the tensor $M_{ij}$ is expressed as the sum of the curvature plus $\frac{1}{2t}$ times Ricci tensor (which is positive). 
Therefore, the Harnack expression $\overline{\Rm}\left( \bar{U}, \bar{U}\right)$ will be strictly positive for sufficiently small time. \\
\\
We now suppose that the Harnack expression $\overline{\Rm}\left( \bar{U}, \bar{U}\right)$ becomes negative later. 
Then, there exists a point $(x_0, \, t_0) \in \bar{M}$ where the Harnack expression becomes zero for the first time. 
We can pick a 2-form $\bar{U}^{ab} \in \wedge ^{2} T^{\ast} _{(x_0,\, t_0)} \bar{M}$, and extend $\bar{U}^{ab}$ in the space-time 
by the following conditions: 
\begin{equation}\label{kaku:1}
\left \{
\begin{alignedat}{1}
&\bar{\nabla} _k\bar{U}^{ij}=\frac{1}{4t}(\bar{\delta} ^{i}_{k}\bar{U}^{0j}-\bar{\delta} ^{j}_{k}\bar{U}^{0i})\\
&\bar{\nabla} _i\bar{U}^{0j}=0\\
&(\bar{\nabla} _0 -\bar{\Delta})\bar{U}^{ij}=0\\
&(\bar{\nabla} _0 -\bar{\Delta})\bar{U}^{0i}=\frac{1}{2t}\bar{U}^{0i},
\end{alignedat}
\right .
\end{equation}
We decide the extension in space by the first and second equations, and in time by the third and forth equations. 
From Lemma 4.1, we have 
\begin{equation}\label{e:1}
\begin{aligned}
(\bar{\nabla} _0-\bar{\Delta})\left( \overline{\Rm}(\bar{U}, \bar{U}) \right) \underset{N}{=}&2R_{ikjl}M^{kl}\bar{U}^{0i}\bar{U}^{0j}-2P_{ikl}P_j^{\,\,l k}\bar{U}^{0i}\bar{U}^{0j}\\
+&8R_{ilkm}P^{l \,\, m}_{\,\,j}\bar{U}^{ij}\bar{U}^{0k}+4R_{imkn}R^{\,\, m\,\, n}_{j\,\, l}\bar{U}^{ij}\bar{U}^{kl}\\
+&[P_{ijm}\bar{U}^{0m}+R_{ijmn}\bar{U}^{mn}][P_{kln}\bar{U}^{0n}+R_{klpq}\bar{U}^{pq}] \, ,
\end{aligned}
\end{equation}
at a point. We recall that the Harnack expression $\overline{\Rm}\left( \bar{U}, \bar{U}\right)$ is weakly positive up to the time $t_0$ , 
so that 
$$\bar{\Delta}\left( \overline{\Rm}(\bar{U}, \bar{U}) \right) \geq 0 \, ,$$ 
at the point $(x_0, \, t_0)$ . \\
\\
On the other hand, the right hand side of the equation $(\ref{e:1})$ is nonnegative as we mentioned in Remark 4.2, 
so that $\bar{\nabla}_0\left( \overline{\Rm}(\bar{U}, \bar{U}) \right)$ must be nonnegative at the point. \\
\\ 
If we get $\bar{\nabla}_0\left( \overline{\Rm}(\bar{U}, \bar{U}) \right) >0$ as well as $\bar{\nabla}_0\left( \overline{\Rm}(\bar{U}, \bar{U}) \right) \geq 0$ , 
the Harnack expression must be negative for a shot time before at the point $x_0$ . This is a contradiction. \\
\\
\begin{remark}
{\rm Note that the result of the formal proof does not mean that we prove the special case of Theorem $\ref{mt:1}$. 
We should modify the following three points in the formal proof so that we rigorously prove Theorem $\ref{mt:1}$: 
\begin{itemize}
\item the manifold $M$ is not always compact. 
\item $(M, g(t))$ has a $weakly$ positive curvature operator.
\item We get the inequality $\bar{\nabla}_0\left( \overline{\Rm}(\bar{U}, \bar{U}) \right) >0$ . 
\end{itemize}
Conversely, if we solve these points, we can finish the proof of Theorem $ref{mt:1}$. 
To solve these points, we will consider the perturbation of $\overline{\Rm}$ in next subsection. }
\end{remark}

\subsubsection{The rigorous proof of Main Theorem}
\begin{proof}[proof of Theorem $\ref{mt:1}$]
We prepare some auxiliary functions constructed by Shi \cite{Sh} for considering perturbation of $\overline{\Rm}$. 
\begin{lemma}[Shi \cite{Sh}]
There exists a smooth function $f\colon M \rightarrow \mathbb{R}$ which satisfies the following properties: 
\begin{itemize}
\item $f(X)\leq 1$ for all $X \in M$, and $f(X)\rightarrow \infty$ as $X\rightarrow \infty$
\item There exists a constant $C>0$ such that $|\nabla ^{(k)} f | \leq C$ for any $k \in \mathbb{N}$. 
\end{itemize}
where $f(X)\rightarrow \infty$ as $X\rightarrow \infty$ means that the set $f\leq C$ is compact for any constant $C$ as a set in space-time. 
\end{lemma}
In the non-compact case, the function $f$ is useful to make maximum principle argument. 
If the manifold is compact, we take $f\equiv 1$. 
By using the function $f$ , we get the following lemma:  
\begin{lemma}[{Hamilton [11, Lemma 5.2.]}]
For any constant $C>0$, any $\eta >0$, and any compact set $K \subset \bar{M}$, 
There exists two functions $\phi$ on the space-time $\bar{M}$ and $\psi$ on $M$ 
such that
\begin{itemize}\label{p:1}
\item $\psi \leq \eta$ for all $t$\, , and $\psi \geq \delta$ for some $\delta >0$\, ,
\item $\phi \geq \eta$ on the compact set $K$\, , and $\phi \geq \varepsilon$ for some  $\varepsilon >0$\, , while $\phi (X, t) \rightarrow \infty$ as $X\rightarrow \infty$\, ,  
\item $(\bar{\nabla} _0-\bar{\Delta} )\phi >C\phi$ , \quad $\bar{\nabla}_0\psi >C\psi$ , \quad $\phi \geq C \psi$ .
\end{itemize}
\end{lemma}
Indeed, we take 
\begin{equation*}
\phi (X, t) =\varepsilon e^{At}f(X)\, , \quad \psi (t)=\delta e^{Bt}\, , 
\end{equation*}
and choose constants $\varepsilon$, $\delta$, $A$, and $B$ such that the following: 
for any constant $C$, any $\eta >0$, and any compact set $K$, 
we can pick $A>C$ and $\varepsilon >0$ such that 
$$\varepsilon \leq \eta e^{-AT}\underset{K}{\max}\,f(X) \, .$$
Furthermore, we can pick $B>C$ and $\delta >0$ such that 
$$\delta < \eta e^{-BT} \quad and \quad \delta <\varepsilon e^{-BT}/C \, . $$
In fact, the first and second properties of the two functions are not necessary in the non-compact case. 
If the manifold is non-compact, we take the limit $\eta \rightarrow 0$ after we finish the maximum principle argument. \\
\\
By using these functions, we consider the perturbation $\widehat{\Rm}$ of $\overline{\Rm}$. 
We define $\widehat{\Rm}$ as follows:
\begin{equation}\label{d:1}
\widehat{\Rm}(\bar{U}, \bar{U})=\bar{R}_{abcd}(\bar{U}, \bar{U})+\frac{1}{t}\phi g_{ij}\bar{U}^{0i}\bar{U}^{0j}+\frac{1}{2}\psi (g_{ik}g_{jl}-g_{il}g_{jk}) \bar{U}^{ij}\bar{U}^{kl}\, .
\end{equation}
Since $\bar{R}_{ijkl}$ is weakly positive, $\bar{R}_{ij0k}$ is bounded, and $\bar{R}_{0i0j}$ is consisted of bounded terms 
and $\frac{1}{2t}R_{ij}$ which is weakly positive, we have 
\begin{equation}
\overline{\Rm}(\bar{U}, \bar{U})\geq -C|U||W|-C|W|^2\, ,
\end{equation}
so that 
\begin{equation*}
\begin{aligned}
\widehat{\Rm}(\bar{U}, \bar{U})&= \overline{\Rm}(\bar{U}, \bar{U})+\frac{1}{t}\phi |W|^2 +\psi |U|^2\\
&\geq \Bigl( \frac{1}{t}\phi -C \Bigl) |W|^2 -C|W||U|+\psi |U|^2\, .
\end{aligned}
\end{equation*}
Hence, if $t>0$ is sufficiently small or the outside of the compact set, $\widehat{\Rm}$ is strictly positive. 
We would like to prove this positivity is preserved for all time $t$. \\
\\
Note that the perturbation $\widehat{\Rm}$ does not take the minimum value out of the compact set in the space-time 
since the term including $\phi$ goes to $\infty$ as $X\rightarrow \infty$, 
so that we do not need to suppose the compactness of the manifold. \\
\\
We now deduce the differential equation for the $\widehat{\Rm}$. 
When the 2-form $\bar{U}^{ab}$ on $\bar{M}$ satisfies 
\begin{equation*}
\begin{aligned}
\bar{\nabla} _i\bar{U}^{0j}=0\, , \quad (\bar{\nabla} _0 -\bar{\Delta})\bar{U}^{ij}=0\, 
\end{aligned}
\end{equation*}
at a point, we have
\begin{equation*}
\begin{aligned}
(\bar{\nabla}_0-\bar{\Delta})\widehat{\Rm}(\bar{U}, \bar{U})
&=(\bar{\nabla}_0 -\bar{\Delta})\overline{\Rm}(\bar{U}, \bar{U})+\frac{1}{t}\Bigl[(\bar{\nabla}_0 -\Delta )\phi-\frac{1}{t}\phi \Bigl] |W|^2\\
+&\frac{1}{t}\phi g_{ij} \bar{U}^{0i}(\bar{\nabla}_0-\Delta )\bar{U}^{0j}+(\bar{\nabla} _0 \psi ) |U|^2-\psi |\bar{\nabla}_k \bar{U}^{ij}|^2
\end{aligned}
\end{equation*}
where $|W|^2=g_{ij}\bar{U}^{0i} \bar{U}^{0j}$ and $|\bar{\nabla}_k \bar{U}^{ij}|^2=g_{im}g_{jl}g_{kn}\bar{\nabla}_k \bar{U}^{ij}\bar{\nabla}_n \bar{U}^{lm}$. 
When the 2-form $\bar{U}^{ab}$ satisfies 
\begin{equation*}
\begin{aligned}
\bar{\nabla} _k\bar{U}^{ij}=\frac{1}{4t}(\bar{\delta} ^{i}_{k}\bar{U}^{0j}-\bar{\delta} ^{j}_{k}\bar{U}^{0i})\, ,\quad (\bar{\nabla} _0 -\bar{\Delta})\bar{U}^{0i}=\frac{1}{2t}\bar{U}^{0i}\, 
\end{aligned}
\end{equation*}
at a point, we have
\begin{equation}\label{d:2}
\begin{aligned}
(\bar{\nabla}_0-\bar{\Delta})\widehat{\Rm}& \geq (\bar{\nabla}_0 -\bar{\Delta})\overline{\Rm}+\frac{1}{t}\Bigl[ (\bar{\nabla}_0 -\Delta) \phi -\frac{C}{t} \psi \Bigl] |W|^2\\
&+ (\bar{\nabla}_0 \psi )|U|^2
\end{aligned}
\end{equation}
from the assumption that the curvature tensor $R_{ijkl}$ is bounded. 
Note that the coefficient of $|W|^2$ of above equation is weakly positive by using $(\ref{p:1})$ and choosing a sufficiently large constant $C>\frac{1}{t}$. \\
\\
Recall that the definition of $\widehat{\Rm}$. We put 
\begin{equation*}
\begin{aligned}
&\hat{R}_{ijkl}=R_{ijkl}+\frac{1}{2}\psi (g_{ik}g_{jl}-g_{il}g_{jk}) \,, \\
&\hat{M}_{ij}=M_{ij}+\frac{1}{t}\phi g_{ij} \, .
\end{aligned}
\end{equation*}
Then, we can write 
\begin{equation}\label{d:3}
\begin{aligned}
&(\bar{\nabla _0}-\bar{\Delta})\widehat{\Rm}(\bar{U}, \bar{U})\\
&\geq 2\hat{R}_{ikjl}\hat{M}^{kl}\bar{U}^{0i}\bar{U}^{0j}-2P_{ikl}P_j^{\,\,l k}\bar{U}^{0i}\bar{U}^{0j}\\
&+8\hat{R}_{ilkm}P^{l \,\, m}_{\,\,j}\bar{U}^{ij}\bar{U}^{0k}+4\hat{R}_{imkn}\hat{R}^{\,\, m\,\, n}_{j\,\, l}\bar{U}^{ij}\bar{U}^{kl}\\
&+[P_{ijm}\bar{U}^{0m}+\hat{R}_{ijmn}\bar{U}^{mn}][P_{kln}\bar{U}^{0n}+\hat{R}_{klpq}\bar{U}^{pq}] \\
&+\frac{1}{t}\Bigl[ (\bar{\nabla}_0 -\Delta) \phi -\frac{C}{t} \psi -C \phi \Bigl] |W|^2+ [\bar{\nabla}_0 \psi -C\psi] |U|^2 \, 
\end{aligned}
\end{equation}
since the manifold has an uniformly bounded curvature. 
Indeed, we have
\begin{equation}\label{eq:876}
\begin{aligned}
(\bar{\nabla _0}-\bar{\Delta})\overline{\Rm}(\bar{U}, \bar{U})&\geq 2\hat{R}_{ikjl}\hat{M}^{kl}\bar{U}^{0i}\bar{U}^{0j}-2P_{ikl}P_j^{\,\,l k}\bar{U}^{0i}\bar{U}^{0j}\\
&+8\hat{R}_{ilkm}P^{l \,\, m}_{\,\,j}\bar{U}^{ij}\bar{U}^{0k}+4\hat{R}_{imkn}\hat{R}^{\,\, m\,\, n}_{j\,\, l}\bar{U}^{ij}\bar{U}^{kl}\\
&+[P_{ijm}\bar{U}^{0m}+\hat{R}_{ijmn}\bar{U}^{mn}][P_{kln}\bar{U}^{0n}+\hat{R}_{klpq}\bar{U}^{pq}] \\
&-\frac{C}{t}(\phi +\psi + \phi \psi ) |W|^2-C\psi |U||W|\\
&-C(\psi ^2 +\psi )|U|^2 \, 
\end{aligned}
\end{equation}
by Lemma 4.1, the definition $(\ref{d:1})$, and the properties $(\ref{p:1})$. 
Here, we simplified these errors, for example, we used $|U|\, |W| \leq |U|^2 +|W|^2$ and $\phi \psi \leq \phi$ since we have $\psi \leq 1$ for small $\eta$. 
We substitute the equation $(\ref{eq:876})$ for the equation $(\ref{d:2})$, so that the equation $(\ref{d:3})$ holds. 
We see that the last two terms of the right hand side of the equation $(\ref{d:3})$ is strictly positive by using $(\ref{p:1})$. \\ 
\\
At last, we use the maximum principle for the differential equation $(\ref{d:3})$ in order to prove Main Theorem. 
Assume that $\widehat{\Rm}({U}, {U})$ is equals to zero at a point $(x_0, t_0)$ of the space-time first where $U=U^{ab} \in \wedge ^2 T_{(x_0, t_0)} ^{\ast} \bar{M}$. 
we can extend $U$ to a 2-form $\bar{U}^{ab}$ on $\bar{M}$, satisfying 
\begin{equation*}
\left \{
\begin{alignedat}{1}
&\bar{\nabla} _k\bar{U}^{ij}=\frac{1}{4t}(\bar{\delta} ^{i}_{k}\bar{U}^{0j}-\bar{\delta} ^{j}_{k}\bar{U}^{0i})\\
&\bar{\nabla} _i\bar{U}^{0j}=0\\
&(\bar{\nabla} _0 -\bar{\Delta})\bar{U}^{ij}=0\\
&(\bar{\nabla} _0 -\bar{\Delta})\bar{U}^{0i}=\frac{1}{2t}\bar{U}^{0i},
\end{alignedat}
\right .
\end{equation*}
at the point $(x_0, t_0) \in \bar{M}$. 
Then, we see that $\bar{\Delta} \widehat{\Rm}$ is nonnegative and the right hand side of the equation $(\ref{d:3})$ is strictly positive at the point. 
Hence, we have $\bar{\nabla}_{0} \widehat{\Rm}(U,U) >0$ at the point from the equation $(\ref{d:3})$. 
This implies $\widehat{\Rm}(U,U)$ must be negative at a short time before. This is a contradiction. 
When we take $\eta \rightarrow 0$, then we have the Harnack inequality. 
\end{proof}

\section{The monotonicity of $\mathcal{W}$-entropy}
In \cite{P1}, Perelman describe that the $\mathcal{W}$-entropy is essentially a total scalar curvature of a hypersurface in spherical thermostat. 
In this section, we verify the argument along [17, Chapter 6]. 
Moreover, we recover the monotonicity of $\mathcal{W}$-entropy from a view point of spherical thermostat. 
Note that the argument in this section includes the heuristic argument in the sense that 
we apply the theory of Riemannian geometry to potentially infinite dimensional manifold 
and we ignore the terms of magnitude $O(1/N)$. \\
\\
The $\mathcal{W}$-entropy is defined by 
\begin{equation*}
\mathcal{W}(g,f,\tau )=\int [\tau (\R+|\nabla f|^2)+f-n] \, (4\pi \tau )^{-\frac{n}{2}}e^{-f}\, dV _g, 
\end{equation*}
where $dV _{g}$ is Riemannian volume form with respect to $g$, and $f$ is a smooth function on $M^n$. \\
\\
In the setting of spherical thermostat as defined section 1, 
we define a diffeomorphism on $\hat{M}$ as follows:
\begin{equation*}
\phi \colon (x^i , y^{\alpha}, \tau) \longrightarrow (x^i ,y^{\alpha} ,\tau (1-\frac{2f}{N}))
\end{equation*}
where $f$ is a function on $\hat{M}$ independent on $\mathbb{S}^N$. 
Then we have
\begin{align*}
\frac{\partial \phi ^0}{\partial \tau}&=1-\frac{2f}{N}-\frac{2\tau}{N}\frac{\partial f}{\partial \tau}, \, \frac{\partial \phi ^0}{\partial x^i}=-\frac{2\tau}{N}\nabla _i f, \,
\frac{\partial \phi ^j}{\partial x^i} =\delta ^j _i , \\
\frac{\partial \phi ^\beta}{\partial y^{\alpha}}&=\delta ^{\beta}_{\alpha},\,
\frac{\partial \phi ^{\beta}}{\partial x^i}= \frac{\partial \phi ^j}{\partial \tau}
= \frac{\partial \phi ^{\beta}}{\partial \tau} =\frac{\partial \phi ^j}{\partial y^{\alpha}}=\frac{\partial \phi ^0}{\partial y^{\alpha}} =0. 
\end{align*}
Hence, we have
\begin{align*}
(\phi ^{\ast} \tilde{g})_{ij}&=\tilde{g}_{ij} \frac{\partial \phi ^i}{\partial x^i} \frac{\partial \phi ^j}{\partial x^j} +\tilde{g}_{00}\frac{\partial \phi ^0}{\partial x^i}\frac{\partial \phi ^0}{\partial x^j}
=\tilde{g}_{ij}+(\frac{N}{2\tau (1-\frac{2f}{N})}+R)(\frac{4\tau ^2}{N^2}\nabla _if\nabla _jf)\\
&=\tilde{g}_{ij}\\
(\phi ^{\ast} \tilde{g})_{\alpha \beta}&=\tau (1-\frac{f}{2N})g_{\alpha \beta}\frac{\partial \phi ^c}{\partial x^{\alpha}}\frac{\partial \phi ^d}{\partial x^{\beta}}=(1-\frac{f}{2N})\tilde{g}_{\alpha \beta}\\
(\phi ^{\ast} \tilde{g})_{00}&=(\frac{N}{2\tau (1-\frac{2f}{N})}+R)(1-\frac{2f}{N}-\frac{2\tau}{N}\frac{\partial f}{\partial \tau})^2=\tilde{g}_{00}-\frac{f}{\tau}-2\frac{\partial f}{\partial \tau}\\
(\phi ^{\ast} \tilde{g})_{i0}&=(\frac{N}{2\tau (1-\frac{2f}{N})}+R)(-\frac{2\tau }{N}\nabla _i f)(1-\frac{2f}{N}-\frac{2\tau}{N}\frac{\partial f}{\partial \tau})=-\nabla _i f\\
(\phi ^{\ast} \tilde{g})_{i \alpha}&=(\phi ^{\ast} \tilde{g})_{\alpha 0}=0.
\end{align*}
up to errors of order $\frac{1}{N}$ where $(\phi ^{\ast} \hat{g})_{ab}=\hat{g}_{cd} \frac{\partial \phi ^c}{\partial x^a}\frac{\partial \phi ^d}{\partial x^b}$. \\
\par
Moreover, Let $\psi _\tau \colon (\tilde{M}, \phi ^{\ast} \tilde{g}) \longrightarrow  (\tilde{M}, \psi _\tau ^{\ast}(\phi ^{\ast} \tilde{g}))$
be an 1-parameter family of diffeomorphisms generated by a vector field $X(\tau)=\nabla f$ and 
$m$ a measure on $\tilde{M}$ satisfies
\begin{equation*}
dm=(4\pi \tau )^{-\frac{n}{2}}e^{-f}dV_{\phi ^{\ast} \tilde{g}}. 
\end{equation*}
If we choose a function $f$ such that $dm$ is independent on $\tau$, 
then the function $f$ satisfies
\begin{equation*}
\frac{\partial f}{\partial \tau}=\frac{1}{2}\tr \frac{\partial \phi ^{\ast} \tilde{g}}{\partial \tau}-\frac{n}{2\tau}.
\end{equation*}
Meanwhile, $g^m:=\psi _\tau ^{\ast}(\phi ^{\ast} \tilde{g})$ satisfies 
\begin{equation}\label{t:1}
\frac{\partial g^m}{\partial \tau}=2\Ric (g^m)+2\Hess_{g^m}(f).
\end{equation}
and a function $f^m:=f\circ \psi _{\tau}$ on $\tilde{M}$ satisfies 
\begin{equation}\label{t:2}
\frac{\partial f^m}{\partial \tau}=\Delta f+R-\frac{n}{2\tau}-|\nabla f|^2
\end{equation}
Hence, under the evolution equations $(\ref{t:1})$, $(\ref{t:2})$, we have
\begin{align*}
g_{00}^m&=(\phi ^{\ast} \tilde{g})_{cd}\frac{\partial \psi _{\tau}^c}{\partial \tau}\frac{\partial \psi _{\tau}^d}{\partial \tau}\\
&=(\phi ^{\ast} \tilde{g})_{00}+(\phi ^{\ast} \tilde{g})_{ij}\frac{\partial \psi _{\tau}^i}{\partial \tau}\frac{\partial \psi _{\tau}^j}{\partial \tau}
+2(\phi ^{\ast} \tilde{g})_{i0}\frac{\partial \psi _{\tau}^i}{\partial \tau}\frac{\partial \psi _{\tau}^0}{\partial \tau}\\
&=(\phi ^{\ast} \tilde{g})_{00}-|\nabla f|^2\\
&=\frac{1}{\tau}(\frac{N}{2}-[\tau (2\Delta f -|\nabla f|^2+\R )+f-n])\\
g^m _{\alpha \beta}&=(1-\frac{2f}{N})\tilde{g}_{\alpha \beta} , \, g^m_{i0}=g^m_{\alpha 0}=g^m_{i\alpha}=0.
\end{align*}
One can see that the integrand of $\mathcal{W}$-entropy appears as a part of $g_{00}^m$. 
To clarify the geometric interpretation of that entropy, 
we consider the hypersurface with respect to $\tau = const$. 
We compute the curvature tensor with respect to the metric on the hypersurface induced by $g^m$, 
\begin{align*}
\R ^m_{ijkl}&=R_{ijkl}\\
\R ^m_{\alpha \beta \gamma \delta}&=R_{\alpha \beta \gamma \delta}-\frac{|\nabla f|^2}{N^2}\tau (g_{\alpha \gamma}g_{\beta \delta}-g_{\beta \gamma}g_{\alpha \delta})\\
\R ^m_{i \alpha j \beta}&=\frac{\tau}{N}g_{\alpha \beta}\nabla _i \nabla _j f
\end{align*}
up to errors of order $\frac{1}{N}$.
Hence,   
\begin{equation}\label{q:2}
\R ^m=\R +\frac{N}{2\tau}+\frac{f}{\tau}-|\nabla f|^2 +2\Delta f=\frac{N}{2\tau}+\frac{1}{\tau}[\tau (2\Delta f -|\nabla f|^2+\R )+f]
\end{equation}
Furthermore, we compute the volume form with respect to $g^m$. 
Let $U_{\alpha}, X_{i}$ be a local coordinate on $\mathbb{S}^N$, $M$ respectively, then we have 
\begin{equation}\label{ss:12}
\begin{aligned}
&\sqrt{\det (g^m_{\alpha \beta})}\sqrt{\det (g^m _{ij})}\prod _{\alpha =1} ^N U_{\alpha}^{\ast}\prod _{i=1}^n X_i ^{\ast}\\
&=(1-\frac{2f}{N})^{\frac{N}{2}}\tau ^{\frac{N}{2}}\sqrt{\det (g_{\alpha \beta})}\sqrt{\det (g _{ij})}\prod _{\alpha =1} ^N U_{\alpha}^{\ast}\prod _{i=1}^n X_i ^{\ast}\\
&=\tau ^{\frac{N}{2}}e^{-f}\sqrt{\det (g_{\alpha \beta})}\sqrt{\det (g _{ij})}\prod _{\alpha =1} ^N U_{\alpha}^{\ast}\prod _{i=1}^n X_i ^{\ast}
\end{aligned}
\end{equation}
up to errors of order $\frac{1}{N}$. 
The third equation is deduced by the binomial theorem with respect to $(1-\frac{2f}{N})^N$ and Taylor expansion with respect to $e^{-2f}$ for large $N$. 
One can see that the volume form with respect to $g^m$ is equal to $\tau ^{\frac{N}{2}}e^{-f}$ times the standard volume form on $M\times \mathbb{S}^N$.
\\
\par
To prove the monotonicity of $\mathcal{W}$-entropy, 
we consider the total scalar curvature of geodesic sphere $\mathbb{S}_{\tilde{M}} (r)$ of radius $r$ on $(\tilde{M}, g^m)$. 
The key of the proof is the following lemma:
\begin{lemma}\label{q:3}
Let $(X,G)$ be a complete Ricci flat Riemannian manifold and $r$ a distance function on $X$. 
Then, 
\begin{align*}
\frac{\partial}{\partial r}\log \int _{\mathbb{S}_X (r)}\R _{\mathbb{S}_X (r)}dS_X\leq \frac{\partial}{\partial r}\log \int _{\mathbb{S}_{\mathbb{R}^n}(r)}\R _{\mathbb{S}_{\mathbb{R}^n(r)}}dS_{\mathbb{R}^n} \, . 
\end{align*}
\end{lemma}
\begin{proof}
Let $h_{ij}$ be the scalar second fundamental form. From Gauss equation (see \cite{CLN}, (1.91)), 
\begin{equation*}
(\Rm _{\mathbb{S}_{X}(r)})_{ijkl}=(\Rm _X)_{ijkl}+h_{il}h_{jk}-h_{ik}h_{jl}\, .
\end{equation*}
Hence, 
\begin{equation*}
(\Ric _{\mathbb{S}_{X}(r)})_{jl}=(\Ric _X)_{jl}-(\Rm _X)_{njln}+h_{il}h_{j}^i-Hh_{jl}\, .
\end{equation*}
where $H:=g^{ij}h_{ij}$ is the mean curvature. Moreover, we have
\begin{equation*}
\R _{\mathbb{S}_X(r)}=\R _X-2(\Ric _X)_{nn}+|h|^2-H^2 \, .
\end{equation*}
Since $X$ is Ricci flat, we have
\begin{equation*}
\R _{\mathbb{S}_X(r)}=|h|^2 -H^2 \, .
\end{equation*}
Hence, we can express the scalar curvature of $\mathbb{S}_X(r)$ by the second fundamental form only. 
In the same way, the scalar curvature of $\mathbb{S}_{\mathbb{R}^n}(r)$ can be expressed in terms of $h$ and $H$ because $\mathbb{R}^n$ is flat. 
We now roughly explain the proof of the Lemma 5.1. 
The second fundamental form is the geometric quantity which expresses the variant of the induced metric 
when the hypersurface move in the direction for the outer normal vector. 
Hence, the variant of the second fundamental of hypersurface for the Ricci flat space with respect to $r$ 
cannot grow larger than Euclidean space's by Bishop-Gromov comparison theorem. 
Moreover, if we regard $h$ and $H$ as a functional for the principal curvature, 
then we see that $|h|^2-H ^2$ is maximum when the hypersurface is totally umbilical hypersurface under the condition where $H$ is constant. 
In fact, the geodesic sphere in Euclidean space is a totally umbilical hypersurface. 
We have thus proved Lemma 5.1. 
\end{proof} 
Note that the  hypersurface with respect to $\tau =$const. 
is equal to the geodesic sphere $\mathbb{S}_{\tilde{M}} (r)$ of radius $r$ on $(\tilde{M}, g^m)$ modulo the magnitude $O(1/N)$. 
Indeed, from the computation of $\ref{ss:12}$, the volume form with respect to $g^{m}$ on the hypersurface is equals to 
$\tau ^{\frac{N}{2}} e^{-f}$ times the standard volume form on $M\times \mathbb{S}^{N}$ up to errors of order $\frac{1}{N}$. 
Since the order of radius $r$ is equals to $\tau ^{\frac{1}{2}}$, We can formally apply Lemma 5.1 to the total scalar curvature with respect to $(\tilde{M}, g^m)$. \\
\\
From $(\ref{q:2})$, 
\begin{align*}
&\int _{\mathbb{S}_{\tilde{M}} (r)}\R ^m dS_{\tilde{M}}\\
&=\frac{N+2n}{2\tau}\int \tau ^{\frac{N}{2}}e^{-f}\sqrt{\det (g_{\alpha \beta})}\sqrt{\det (g _{ij})}\prod _{\alpha =1} ^N U_{\alpha}^{\ast}\prod _{i=1}^n X_i ^{\ast}
+(4\pi )^{\frac{n}{2}} \tau ^{\frac{N}{2}+\frac{n}{2}-1}\mathcal{W}.\\
&=(C_1(N,n)+(4\pi )^{\frac{n}{2}}\mathcal{W})\tau^{\frac{N}{2}+\frac{n}{2}-1}.
\end{align*}
where $C_i (N,n)$ are constants depending on $N$ and $n$. 
On the other hand, we have
\begin{align*}
\int _{\mathbb{S}_{\mathbb{R}^{n+N+1}}(r)}\R _{\mathbb{S}_{\mathbb{R}^{n+N+1}(r)}}dS _{\mathbb{R}^{n+N+1}}&=\frac{(N+n)(N+n-1)}{r^2}\int dS\\
&=C_2(N,n)r^{N+n-2}\, . 
\end{align*}
Hence, we see that the $\mathcal{W}$-entropy is increasing for $\tau$ by formally applying the above lemma to $(\tilde{M}, g^m)$ 
which is potentially infinite dimension. 


\begin{thebibliography}{9}
\bibitem{BW}
C. B\"{o}hm and B. Wilking, Manifolds with positive curvature operators are space forms. Ann. of Math. (2) 167 (2008), no. 3, 1079-1097.
\bibitem{BR}
S. Brendle, A generalization of Hamilton's differential Harnack inequality for the Ricci flow. J. Differential Geom. 82 (2009), 207-227.
\bibitem{TOP2}
E. Cabezas-Rivas and P. Topping, The canonical expanding soliton and Harnack inequalities for Ricci flow. Trans. Amer. Math. Soc. 364 (2012), 3001-3021. 
\bibitem{TOP3}
E. Cabezas-Rivas and P. Topping, The canonical shrinking soliton associated a the Ricci flow. Calc. Var. and PDE, 43 (2012), 173-184. 
\bibitem{CC}
B. Chow and S. Chu, A geometric interpretation of Hamilton's Harnack inequality for the Ricci flow. Math. Res. Lett. 2 (1995), 701-718. 
\bibitem{CK}
B. Chow and D. knopf, New Li-Yau-Hamilton inequalities for the Ricci Flow via the space-time approach. J. Differential Geom. 60 (2002), 1-54.
\bibitem{CLN}
B. Chow, P. Lu, and L. Ni, Hamilton's Ricci Flow. AMS and Science Press (2008).
\bibitem{De}
D. DeTurck, Deforming metrics in the direction of their Ricci tensors. In Collected papers on Ricci flow. Edited by H. D. Cao, B. Chow, S. C. Chu and S. T. Yau. 
Series in Geometry and Topology, 37. International Press (2003).
\bibitem{H5}
R. Hamilton, Three-manifolds with positive Ricci curvature, J. Differential Geom. 17 (1982) 255-306. 
\bibitem{H3}
R. Hamilton, The Ricci flow on surfaces. Contemp. Math. 71 (1988) 237-261. 
\bibitem{H1}
R. Hamilton, The Harnack estimate for the Ricci flow. J. Differential Geom. 37 (1993) 225-243.
\bibitem{H2}
R. Hamilton, The Harnack estimate for the mean curvature flow, J. Differential Geom. 41 (1995) 215-226. 
\bibitem{H4}
R. Hamilton, The formation of singularities in the Ricci flow. Surveys in differential geometry, Vol. II (Cambridge, MA, 1993) 7-136, Internat. Press, Cambridge, MA (1995).
\bibitem{Kaz}
J. Kazdan, Another proof of Bianchi's identity in Riemannian geometry, Proc. Am. Math. Soc.  81 (1981) 341-342.
\bibitem{KL}
B. Kleiner and J. Lott, Notes on Perelman's papers. Geometry and Topology, 12 (2008) 2587-2855. 
\bibitem{LY}
P. Li and S-T. Yau, Inequalities for the parabolic Schrodinger operator, Acta Math. 156 (1986)  153-201.  
\bibitem{P1}
G. Perelman, The entropy formula for the Ricci flow and its geometric applications, math.DG/0211159v1
\bibitem{TH}
W. Thurston, Three dimensional manifolds, Kleinian Groups and Hyperbolic Geometry. 
Bull. Amer. Math. Soc. 6 (1982) 357-381.  
\bibitem{Sh}
W. Shi, Ricci deformation of the metric on complete noncompact Riemannian manifolds. J. Differential Geom. 30 (1989) 303-394. 
\bibitem{TOP}
P. Topping, Lectures on the Ricci flow. L.M.S. Lecture notes series 325 C.U.P (2006). 
\bibitem{TOP4}
P. Topping, $\mathcal{L}$-optimal transportation for Ricci flow. J. Reine Angew. Math. 636  (2009) 93-122.
\end{thebibliography}
\end{document}